\documentclass[12pt]{article}
\usepackage{epsfig,amsmath,amsthm,amssymb,graphics,amsfonts,psfrag}
\usepackage[letterpaper,margin=0.9in]{geometry}

\theoremstyle{plain}
\newtheorem{thm}{Theorem}[section]

\newtheorem{prop}[thm]{Proposition}

\theoremstyle{definition}
\newtheorem{defn}[thm]{Definition}
\newtheorem{ex}[thm]{Example}

\def\R{\mathbb{R}}
\def\C{\mathbb{C}}

\def\P{\mathbb{P}}
\def\E{\mathbb{E}}
\def\Z{\mathbb{Z}}
\def\I{\infty}

\newcommand{\be}{\begin{equation}}
\newcommand{\ee}{\end{equation}}
\newcommand{\bea}{\begin{eqnarray}}
\newcommand{\eea}{\end{eqnarray}}
\newcommand{\beann}{\begin{eqnarray*}}
\newcommand{\eeann}{\end{eqnarray*}}
\newcommand{\benn}{\begin{equation*}}
\newcommand{\eenn}{\end{equation*}}

\def\ra{\rightarrow}
\def\I{\infty}

\begin{document}
 
\title{A Mathematical Framework for Critical Transitions: Bifurcations, Fast-Slow Systems and Stochastic Dynamics}
\author{Christian Kuehn\thanks{Center for Applied Mathematics, Cornell University}}

\maketitle 

\begin{abstract}
Bifurcations can cause dynamical systems with slowly varying parameters to transition to far-away attractors. The terms ``critical transition'' or ``tipping point'' have been used to describe this situation. Critical transitions have been observed in an astonishingly diverse set of applications from ecosystems and climate change to medicine and finance. The main goal of this paper is to give an overview which standard mathematical theories can be applied to critical transitions. We shall focus on early-warning signs that have been suggested to predict critical transitions and point out what mathematical theory can provide in this context. Starting from classical bifurcation theory and incorporating multiple time scale dynamics one can give a detailed analysis of local bifurcations that induce critical transitions. We suggest that the mathematical theory of fast-slow systems provides a natural definition of critical transitions. Since noise often plays a crucial role near critical transitions the next step is to consider stochastic fast-slow systems. The interplay between sample path techniques, partial differential equations and random dynamical systems is highlighted. Each viewpoint provides potential early-warning signs for critical transitions. Since increasing variance has been suggested as an early-warning sign we examine it in the context of normal forms analytically, numerically and geometrically; we also consider autocorrelation numerically. Hence we demonstrate the applicability of early-warning signs for generic models. We end with suggestions for future directions of the theory.  
\end{abstract}


\section{Introduction}
\label{sec:intro}

In this paper ``critical transitions'' or ``tipping points'' are viewed from the perspective of dynamical systems. Our aim is to point out that various observations, assumptions and ideas developed in diverse scientific disciplines can be expressed naturally using standard mathematical theory. In particular, we hope that this paper can be viewed as a mathematical complement to the excellent review by Scheffer et al \cite{Schefferetal}. A non-mathematical working definition of a critical transition is an abrupt change in a dynamical system. To illustrate the concept we list four examples

\begin{itemize}
 \item In ecosystems rapid changes to desertification or extinctions of species can occur \cite{Schefferetal1,SchefferCarpenter}.
 \item Medical conditions can quickly change from regular to irregular behavior; examples are asthma attacks \cite{Venegasetal} or epileptic seizures \cite{McSharrySmithTarassenko}. 
 \item Financial markets can transition from a balanced market to a financial crisis \cite{MayLevinSugihara}. 
 \item Changes in the climate and its constituent subsystems can occur abruptly \cite{Bakkeetal,Lentonetal,Alleyetal}. 
\end{itemize}
 
It is clear that we would like to understand and predict these phenomena. At first glance it might be surprising that all four examples have anything in common as they arise in completely different contexts and situations. Nevertheless, it has become apparent that critical transitions share several attributes \cite{Scheffer,Schefferetal}:

\begin{itemize}
 \item[(1)] An abrupt qualitative change in the dynamical system occurs.
 \item[(2)] The change occurs rapidly in comparison to the regular dynamics.
 \item[(3)] The system crosses a special threshold near a transition. 
 \item[(4)] The new state of the system is far away from its previous state.
\end{itemize} 

Furthermore, significant progress has been made in predicting a critical transition before it occurs. The goal is to infer from previous data when a catastrophic shift in the dynamics is going to occur. Ideally we would like to have a comprehensive list of early-warning signs. A variety of system-specific criteria could be introduced; but we are more interested in generic indicators that are expected to be applicable to large classes of transitions. The following assumption will be of major importance \cite{Schefferetal}

\begin{itemize}
 \item[(5)] There is small noise in the system i.e. the data has a major deterministic component with small ``random fluctuations''.
\end{itemize}

There are several characteristics that have been observed in systems before critical transitions. We shall only list a few of them here: 

\begin{itemize}
 \item[(6)] The system recovers slowly from perturbations (``slowing down'').
 \item[(7)] The variance of the system increases as the transition is approached.
 \item[(8)] The noisy fluctuations become more asymmetric. 
 \item[(9)] The autocorrelation increases before a transition.
\end{itemize}  

Figure \ref{fig:fig12} shows time series with critical transitions; the times series have been generated by simulating two generic models discussed in Sections \ref{sec:moments}-\ref{sec:autocorr} for the fast-slow fold and transcritical bifurcations. Many natural questions arise regarding analysis and comparison of these two time series. The observations (1)-(9) are extremely important for this purpose. However, it is desirable to embed these observations into a mathematically precise description of the system dynamics and to identify them in generic models. Relations to bifurcation theory and some indicators have been partially analyzed using statistical techniques such as autoregressive models \cite{Schefferetal}.\\ 

\begin{figure}[htbp]
\psfrag{x}{$x$}
\psfrag{t}{$t$}
\psfrag{a}{(a)}
\psfrag{b}{(b)}
	\centering
		\includegraphics[width=0.95\textwidth]{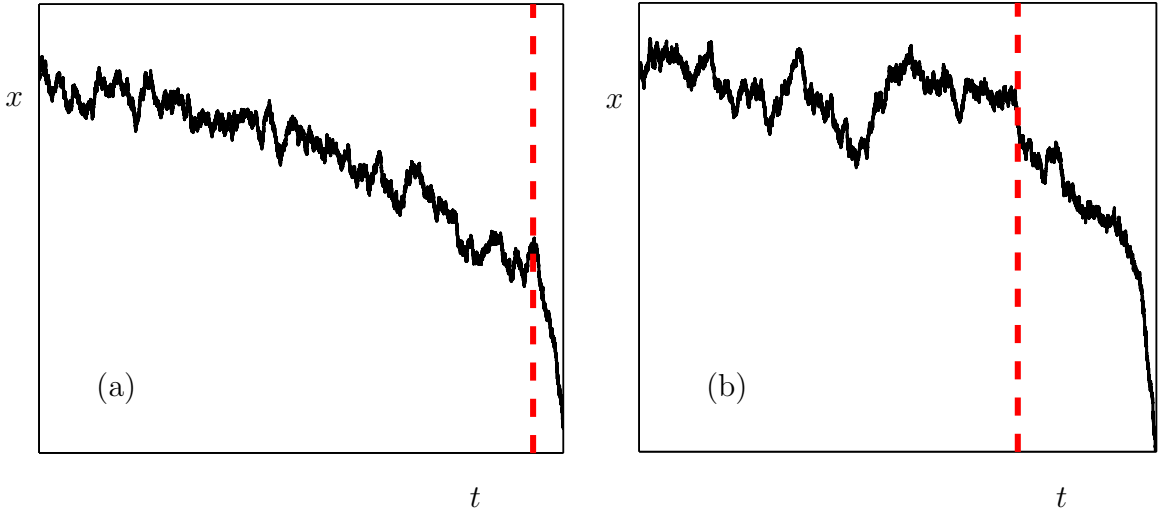}
	\caption{\label{fig:fig12} Two time series with critical transitions. The red dashed vertical lines have been added to indicate where a clear visual change in the time series behavior appears; both series have been generated using fast-slow stochastic dynamical systems: (a) fold and (b) transcritical. The time series for the generic models we propose resemble time series from experiments.}	
\end{figure} 

The major goal of the current paper is to argue that the observations (1)-(9), that are usually made in applications, should be understood from a mathematical viewpoint by using deterministic and stochastic multiscale dynamical systems. We shall try to provide an overview which mathematical concepts and tools can be used to develop a theory of critical transitions. The natural starting point is bifurcation theory \cite{GH,Kuznetsov} and in this respect our approach is closest to recent work by Sieber and Thompson \cite{ThompsonSieber,ThompsonSieber1} that realizes the need for a detailed bifurcation-theoretic analysis of tipping points. Given a dynamical system, such as a differential equation or iterated map, bifurcation theory can be used to classify qualitative transitions under the variation of parameters. It has been successfully applied in fields ranging from physics, engineering and chemistry to modern developments in neuroscience and mathematical biology \cite{Strogatz}. Here we shall focus on differential equations to simplify the discussion but remark that discrete time systems can also be studied from this perspective. As a second step we introduce stochasticity into the dynamics. Our approach is closest to the work by Berglund and Gentz \cite{BerglundGentz4,BerglundGentz5}; they demonstrated the applicability of stochastic multiscale differential equations in a variety of contexts such as climate modeling and neuroscience. Here we point out what their approach implies for critical transitions. We also use numerical simulation of ``normal-form''-type models and compare our results to theoretic results obtained from Fokker-Planck equations. Our numerical approach can also provide benchmark data for time series analysis methods \cite{LivinaLenton,KleinenHeldPetschel-Held}.\\

The structure of the paper is as follows. In Section \ref{sec:det1} we explain why the differential equations describing critical transitions should have multiple time scales. The focus will be on two scales described by a fast-slow system. We suggest that fast-slow systems theory provides a natural definition for a critical transition and check that certain bifurcations satisfy this definition. Furthermore we review the basic calculation for ``slowing down'' and point that slowing down differs for different types of critical transitions. In Section \ref{sec:det2} we review the concept of normal hyperbolicity that separates regular fast-slow system dynamics from dynamic bifurcations; the application in the context of critical transitions is indicated. In Section \ref{sec:stochastic_basic} we recall some basic tools and viewpoints in stochastic analysis. Stochastic fast-slow systems are introduced and some recent results are stated. The case of a fast-slow stochastic system away from a critical transition is defined and its properties are investigated; this provides a comparative method to detect transitions and also to estimate system parameters. In Section \ref{sec:indicators} we review recent progress in stochastic bifurcation theory and show how this theory should apply to and interact with a critical transition involving noise. The concept of stochastic bifurcation is much less developed. However, we introduce an example that shows that so-called P-bifurcations can be early-warning signs. In Section \ref{sec:noiseinduced} noise-induced phenomena are discussed that have recently been discovered in many mathematical models. The prediction of critical transitions is more complicated in this context and a numerical example illustrates this point. In Section \ref{sec:moments} we consider critical transitions under the simplest mathematical assumptions. We calculate the variance of distributions of trajectories near a stochastic critical transition point using a Fokker-Planck approach. In Section \ref{sec:simulation} we use numerical simulations to expand on our modeling approach and we discuss the variance as an indicator if only a single sample path is available. Section \ref{sec:autocorr} extends the numerical simulation approach to autocorrelation. Section \ref{sec:discussion} concludes the paper with a discussion of the current state of the theory and a discussion of topics we omitted. We also sketch some directions for future work.

\section{Fast-Slow Systems I: Critical Transitions}
\label{sec:det1} 

We start with the deterministic theory. Our first goal is to make the term ``critical transition'' mathematically more precise. Consider the parametrized family of ordinary differential equation:
\be
\label{eq:det_gen}
\frac{dx}{dt}=x'=f(x;y)
\ee
where $x\in \R^m$ are phase space variables and $y\in \R^m$ represent parameters. A general statement that can often be found in the description of critical transitions in applications is that ``a parameter evolves slowly until the tipping point is reached''. Therefore it is a natural approach to include the parameters into the original differential equation. The parametrized family \eqref{eq:det_gen} can be written as
\be
\label{eq:fss}
\begin{array}{lcl}
x' & = & f(x,y),\\
y' & = & 0.
\end{array}
\ee 
Using \eqref{eq:fss} it is easy to incorporate slowly varying parameters by adding a slow evolution to $y$
\be
\label{eq:fs_gen}
\begin{array}{lcr}
x' & = & f(x,y),\\
y' & = & \epsilon g(x,y),
\end{array}
\ee 
where $0<\epsilon\ll 1$ is a small parameter and $g$ is assumed to be sufficiently smooth. In many cases it suffices to assume that the parameter dynamics is de-coupled from phase space dynamics and one assumes $g\equiv 1$. The ODEs \eqref{eq:fs_gen} form a fast-slow system where the variables $x\in \R^m$ are the fast variables and $y\in \R^m$ are the slow variables. The parameter $\epsilon$ describes the time scale separation. We point out that the ``inclusion of dynamic slow parameters'' is entirely standard and well-known in the theory of multiple time scale dynamics.\\ 

\textit{Remark:} A new introductory book to fast-slow systems is currently being written \cite{KuehnBook}. The book is going to include the deterministic theory as well as numerical and stochastic components that are relevant in Sections \ref{sec:stochastic_basic}-\ref{sec:simulation}. Classical references for deterministic fast-slow systems are \cite{Jones,MisRoz,Grasman}. In the current paper we restrict ourselves to review the necessary definitions and concepts that can directly be applied to critical transitions.\\

Equation \eqref{eq:fs_gen} can be re-written by changing from the fast time scale $t$ to the slow time scale $\tau=\epsilon t$
\be
\label{eq:fs_gen1}
\begin{array}{rcrcl}
\epsilon \frac{dx}{d\tau} & = & \epsilon \dot{x} & = & f(x,y),\\
\frac{dy}{d\tau} & = & \dot{y} & = & g(x,y).
\end{array}
\ee 
The first step to analyze a fast-slow system is to consider the singular limit $\epsilon\ra 0$. In the formulation \eqref{eq:fs_gen} this yields the parametrized family \eqref{eq:fss} which is also known as the fast subsystem or layer equations. Considering the singular limit in \eqref{eq:fs_gen1} gives the slow subsystem or reduced system
\be
\label{eq:sss}
\begin{array}{lcl}
0 & = & f(x,y),\\
\dot{y} & = & g(x,y).
\end{array}
\ee
The associated subsystem flows are naturally called the fast flow and the slow flow. Equation \eqref{eq:sss} is a differential-algebraic equation so that the slow flow is constrained to
\benn
C=\{(x,y)\in\R^{m+n}:f(x,y)=0\}.
\eenn 
The set $C$ is called the critical set or the critical manifold if $C$ is manifold. The points in $C$ are equilibria for the fast subsystem \eqref{eq:fss}. $C$ is normally hyperbolic at $p\in\R^{m+n}$ if the matrix $(D_xf)(p)$ is hyperbolic i.e. all its eigenvalues have non-zero real parts. If all eigenvalues have negative/positive real parts then $C$ is attracting/repelling at $p$; if $C$ is normally hyperbolic and neither attracting nor repelling we say it is of saddle-type. For a normally hyperbolic critical manifold the implicit function theorem gives
\benn
C=\{(x,y)\in\R^{m+n}:h_0(y)=x\}
\eenn
where $h_0:\R^n\ra \R^m$ satisfies $f(h_0(y),y)=0$. Then the slow flow can be written as
\benn
\dot{y}=g(h_0(y),y).
\eenn

Fenichel's Theorem \cite{Fenichel4,WigginsIM,Tikhonov} provides a complete description of the dynamics for normally hyperbolic invariant manifolds.

\begin{thm}[Fenichel's Theorem]
\label{thm:fenichel1}
Suppose $S=S_0$ is a compact normally hyperbolic submanifold (possibly with boundary) of the critical manifold $C$. Then for $\epsilon>0$ sufficiently small there exists a locally invariant manifold $S_\epsilon$ diffeomorphic to $S_0$. $S_\epsilon$ has a Hausdorff distance of $O(\epsilon)$ from $S_0$ and the flow on $S_\epsilon$ converges to the slow flow as $\epsilon \to 0$.
\end{thm}

$S_\epsilon$ is called a slow manifold. Different slow manifolds $S_\epsilon$ lie at a distance $O(e^{-K/\epsilon})$ from each other and so we will often simply refer to ``the'' slow manifold as the choice of representative is irrelevant for many asymptotic results. A normally hyperbolic critical manifold $C_0$ has associated local stable and unstable manifolds
\beann
W^s(C_0) = \bigcup_{p \in C_0} W^s(p), \qquad \text{and} \qquad W^u(C_0) = \bigcup_{p \in C_0} W^u(p),
\eeann
where $W^s(p)$ and $W^u(p)$ are the local stable and unstable manifolds of $p$ as a hyperbolic equilibrium of the fast subsystem. These manifolds also persist for $\epsilon>0$ sufficiently small. In addition to Fenichel's Theorem there are coordinate changes that simplify a fast-slow system considerably near a critical manifold \cite{Fenichel4,JonesKaperKopell} if the slow flow has no bounded invariant sets.

\begin{thm}[Fenichel Normal Form]
\label{thm:FNform}
Suppose $S_0$ is a compact normally hyperbolic submanifold of $C$ with $m_u$ unstable and $m_s$ stable fast directions and that the slow flow is rectifiable on $S_0$. Then there exists a smooth invertible coordinate change $(x,y)\mapsto (a,b,v)\in\R^{m_u}\times \R^{m_s}\times \R^n$ so that a fast-slow system \eqref{eq:fs_gen} can be written as:
\bea
\label{eq:FNform1}
a'&=& \Lambda(a,b,v,\epsilon)a\nonumber,\\
b'&=& \Gamma(a,b,v,\epsilon)b,\\
v'&=& \epsilon(e_1+H(a,b,v,\epsilon)ab)\nonumber,
\eea
where $\Lambda$, $\Gamma$ are matrix-valued functions. $\Lambda$ has $m_u$ positive and $\Gamma$ has $m_u$ negative eigenvalues, $e_1=(1,0,\ldots,0)^T\in \R^n$ is a unit vector and $H$ is bilinear in $a,b$.
\end{thm}

The manifold $S_0$ perturbs to a slow manifold $S_\epsilon$ by Fenichel's Theorem. Then this slow manifold is ``straightened'' together with its stable and unstable manifolds that become coordinate planes \cite{Jones}. Therefore we can basically assume that the fast subsystem near a normally hyperbolic critical manifold is linear with eigendirections aligning with the coordinates. This will provide the basis for our discussion of dynamical behavior far away from critical transition points. The next classical example illustrates the definitions and shows how normal hyperbolicity can fail.

\begin{ex}
\label{ex:fold}
Consider a planar fast-slow system modeling a fold bifurcation with slow parameter drift \cite{KruSzm1}:
\be
\label{eq:fp_fs}
\begin{array}{rcl}
\epsilon \dot{x} & = & -y-x^2,\\
\dot{y} & = & 1.
\end{array}
\ee

\begin{figure}[htbp]
\psfrag{x}{$x$}
\psfrag{gamma}{$\gamma_0$}
\psfrag{gammaeps}{$\gamma_\epsilon$}
\psfrag{y}{$y$}
\psfrag{a}{(a)}
\psfrag{b}{(b)}
\psfrag{eps23}{$\epsilon^{2/3}$}
\psfrag{C}{$C$}
	\centering
		\includegraphics[width=0.85\textwidth]{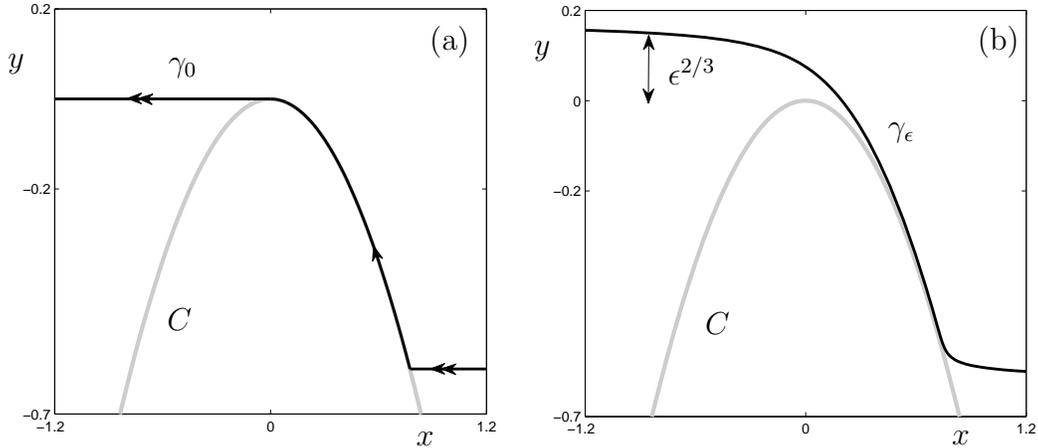}
	\caption{\label{fig:fig4}Illustration for \eqref{eq:fp_fs}. (a) Singular limit $\epsilon=0$ with a candidate trajectory $\gamma_0$ consisting of two fast and one slow segment is shown. (b) Trajectory $\gamma_\epsilon$ for \eqref{eq:fp_fs} with $\epsilon=0.02$ and initial condition $(x(0),y(0))=(1.2,-0.6)$.}	
\end{figure} 

The critical manifold $C=\{(x,y)\in\R^2:y=-x^2\}$ is normally hyperbolic away from the fold bifurcation point $(x,y)=(0,0)$ of the fast subsystem; the point $(0,0)\in C$ is also refered to as a fold point. Observe that the set $C^a:=C\cap\{x>0\}$ are attracting equilibrium points for the fast subsystem while points on $C^r:=C\cap\{x<0\}$ are repelling; see Figure \ref{fig:fig4}. To derive an expression for the slow flow one can differentiate $y=-x^2$ implicitly with respect to $\tau$ giving $\dot{y}=-2x\dot{x}$ which yields
\benn
\dot{x}=-\frac{1}{2x}.
\eenn
Note that the slow flow is not well-defined at $x=0$. However, by rescaling of time $\tau\ra 2x\tau$, which reverses the direction of trajectories on $C^r$, the slow flow can be desingularized. The flow of \eqref{eq:fp_fs} for $\epsilon=0$ can be described by combining trajectories of the fast and slow subsystems; see Figure \ref{fig:fig4}(a). A solution starting in $W^s(C^a)$ approaches it rapidly, then it follows the slow flow on $C$ and finally ``jumps'' at the fold bifurcation point toward $x=-\I$.
\end{ex}

Recall that assumption (1) in Section \ref{sec:intro} requires a critical transition to occur at a point when there is a sudden change from slow dynamics to a fast repelling segment. To make this idea more precise we recall one more definition from fast-slow systems. In the singular limit $\epsilon=0$ trajectories can be considered as concatenations of trajectory segments of the fast and slow subsystems. A candidate \cite{Benoit2,Haiduc} is defined as a homeomorphic image $\gamma_0(t)$ of a real interval $(a,b)$ with $a < b$ where

\begin{itemize}
 \item the interval is partitioned as $a = t_0 < t_1 < \cdots < t_m = b$,
 \item the image of each subinterval $\gamma_0(t_{j-1},t_j)$ is a trajectory of either the fast or the slow subsystem,
 \item and the image $\gamma_0(a,b)$ has an orientation that is consistent with the orientations on each subinterval $\gamma_0(t_{j-1},t_j)$ induced by the fast and slow flows.  
\end{itemize}

Note that we can also view a candidate as a trajectory of a hybrid system. If consecutive images $\gamma_0(t_{j-1},t_j)$ and $\gamma_0(t_{j},t_{j+1})$ are trajectories for different subsystems, i.e. there is a transition at $t_j$ from fast to slow or from slow to fast, then we say that $\gamma_0(t_j)$ is a transition point. Using candidates and transition points we can easily give a rigorous definition of critical transitions.

\begin{defn}
\label{defn:ct}
Let $p=(x_p,y_p)\in C$ be a point where the critical manifold $C$ is not normally hyperbolic. We say that $p$ is a critical transition if there is a candidate $\gamma_0$ so that 
\begin{itemize}
 \item [(C1)] $\gamma_0(t_{j-1},t_j)$ is a normally hyperbolic attracting submanifold of $C$,
 \item [(C2)] $p=\gamma_0(t_j)$ is a transition point,
 \item [(C3)] and $\gamma_0(t_{j-1},t_j)$ is oriented from $\gamma_0(t_{j-1})$ to $\gamma_0(t_j)$. 
\end{itemize}   
\end{defn}  

From a fast-slow systems perspective a critical transition occurs at a bifurcation point $y=y_p$ of the fast subsystem that induces switching from a stable slow motion to a fast motion. Definition \ref{defn:ct} can easily be generalized to more complicated invariant sets of the fast subsystem. For example, if $p$ is a point that lies on a family of fast subsystem periodic orbits we can again define a slow flow by averaging over the periodic orbit \cite{KuehnSlowExit,BerglundGentz}. Then the same definition applies. The next steps are straightforward checks whether several classical bifurcations are critical transitions. From now on we shall restrict ourselves to the case $n=1$ reflecting slow variation of one parameter; without loss of generality we can assume that the bifurcation point of the fast subsystem is located at $(x,y)=(0,0)$ and we assume that $g\equiv 1$ near the origin to simplify the exposition (in principle we will only need that $g(x,y)\geq K>0$ near the origin for some constant $K$ independent of $(x,y)$).
 
\begin{prop}
Suppose $m=1$ so that \eqref{eq:det_gen} is 1-dimensional and that there is a generic fold (or saddle-node) bifurcation at $y=0$. Then the fold bifurcation is also a critical transition.
\end{prop}

\begin{proof}
Near a generic fold bifurcation the flow is topologically conjugate to the normal form \cite{GH} of a fold bifurcation  
\beann
x'&=&-y-x^2,\\
y'&=&\epsilon.
\eeann
The critical manifold is $C=\{y=-x^2\}$ and $C^a:=C\cap \{x>0\}$ is normally hyperbolic and attracting. Then the candidate
\benn
\gamma_0=C^a\cup \{[0,-\I),\times\{0\}\}
\eenn
shows that the fold bifurcation is a critical transition.
\end{proof}

The main idea for the fold bifurcation and all the other bifurcations discussed below is illustrated in Figure \ref{fig:fig3}.

\begin{figure}[htbp]
\psfrag{x}{$x$}
\psfrag{x1}{$x_1$}
\psfrag{y}{$y$}
\psfrag{a}{(a) fold bifurcation}
\psfrag{b}{(b) Hopf bifurcation}
\psfrag{c}{(c) pitchfork bifurcation}
\psfrag{d}{(d) transcritical bifurcation}
	\centering
		\includegraphics[width=0.8\textwidth]{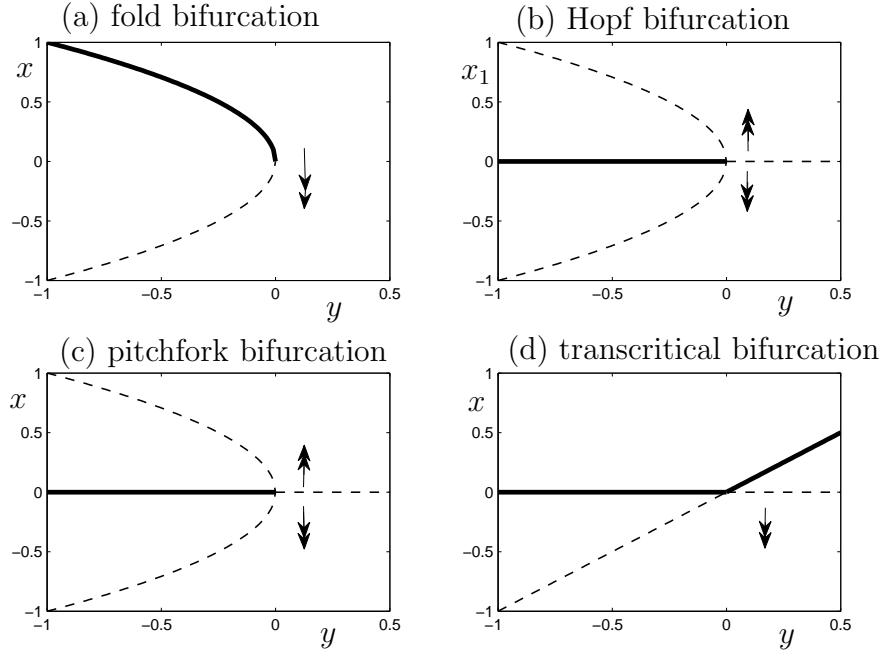}
	\caption{\label{fig:fig3} Fast subsystem bifurcation diagrams for four bifurcations that are critical transitions. Solid curves indicate stability, dashed curves instability. For the Hopf bifurcation in (b) only the projection onto $(x_1,y)$ is shown. Double arrows indicate the flow of the fast subsystem.}	
\end{figure} 

\begin{prop}
Suppose $m=2$ so that \eqref{eq:det_gen} is a planar system. Suppose there is a generic Hopf bifurcation of the fast subsystem at $y=0$ with first Lyapunov coefficient $l_1\neq 0$. The Hopf bifurcation is a critical transition if it is subcritical ($l_1>0$). If it is supercritical ($l_1<0$) then the transition is not critical. 
\end{prop}

\begin{proof} 
By genericity of the Hopf bifurcation we can consider the normal form
\be
\label{eq:Hopf_Euclidean}
\begin{array}{lcl}
x_1' & = & yx_1-x_2+l_1x_1(x_1^2+x_2^2),\\
x_2' & = & x_1+yx_2+l_1x_2(x_1^2+x_2^2),\\
y'&=& \epsilon.\\
\end{array}
\ee 
The equilibrium point $x^*=0$ of the fast subsystem is stable for $y<0$ and loses stability at $y=0$ as a pair of complex conjugate eigenvalues of $(D_xf)(0,y)$ passes through the imaginary axis at $y=0$. Suppose first that $l_1>0$ and consider the candidate
\benn
\gamma_0=\{x=0,y<0\}\cup \mathcal{S}
\eenn
where $\mathcal{S}$ is a spiral trajectory lying in fast subsystem unstable manifold of $(x_1,x_2)=(0,0)$ at $y=0$; this concludes the first part of the proof. If $l_1<0$ then the fast subsystem Hopf bifurcation is supercritical so that there are stable periodic orbits of amplitude $\sqrt{y}$ for $y>0$. Suppose there exists a candidate $\gamma_0$ through $(x,y)=(0,0)$ that satisfies (C1)-(C3) with $\gamma_0(t_j)=(0,0)$. Note that by (C1) we must have that $\gamma_0(t_{j-1},t_j)$ is contained in $\{y<0,x_1=0=x_2\}$. By (C2) we note that $\gamma_0(t_j,t_{j+1})$ cannot be contained in the bifurcating family of periodic orbits or in $\{y>0,x_1=0=x_2\}$. Since $(0,0)=(x_1,x_2)$ is asymptotically stable as an equilibrium point for the fast subsystem we can conclude that (C2) can never be satisfied for any candidate.  
\end{proof}

The fold and Hopf bifurcations are the only generic bifurcations occurring in one-parameter families of equilibrium points of flows \cite{Kuznetsov}. Under additional assumptions on the structure of the equations (e.g. assuming symmetries) one also often considers the following two one-parameter bifurcations:
\benn
\begin{array}{lcll}
x' & = & yx+\alpha x^3 & \qquad \text{pitchfork bifurcation,}\\
x' & = & yx-x^2 & \qquad \text{transcritical bifurcation.} \\
\end{array}
\eenn 
The analysis of the pitchfork bifurcation is completely analogous to the Hopf bifurcation case. Indeed, recall that the Hopf normal form \eqref{eq:Hopf_Euclidean} can be transformed into polar coordinates $(r,\theta)\in(\R^+,S^1)$
\beann
r'&=&yr+l_1r^3,\\
\theta'&=&1.
\eeann

\begin{prop}
The pitchfork bifurcation is a critical transition when it is subcritical ($\alpha>0$) and it is not a critical transition if it is supercritical ($\alpha<0$).
\end{prop}

The transcritical bifurcation case is slightly more interesting.

\begin{prop}
The transcritical bifurcation is a critical transition.
\end{prop}

\begin{proof}
We can again work with the normal form. Consider the candidate
\benn
\gamma_0=\{x=0,y<0\}\cup \{x<0,y=0\}
\eenn
which shows that we have a critical transition.
\end{proof}

Note carefully that there is a candidate trajectory for the transcritical case that has $(0,0)$ as a transition point but that does not satisfy (C1). In particular, not all fast segments escape at a critical transition. Furthermore, the assumption that $g\neq 0$ near the origin does suffice to make the transcritical bifurcation a critical transition. The fast-slow structure naturally suggests additional quantitative measures for critical transitions; see e.g. Definitions \ref{defn:recovery} and \ref{defn:jump} below. Furthermore our definition can easily be extended to any possible bifurcation scenario in a fast-slow system including higher codimension bifurcations and global bifurcations e.g. the lists of bifurcations in \cite{ThompsonSieber} can be subsumed into a fast-slow systems framework.\\ 

Critical slowing down is an indicator of how far we are away from a critical transition point \cite{vanNesScheffer}. Recall that if we have a stable solution point $X=X(t)$ for \eqref{eq:det_gen} and want to consider the evolution of perturbations $X+u$ for $\|u\|$ sufficiently small then 
\be
\label{eq:variational}
u'=f(X+u)-f(X)\approx (D_xf)(X)u
\ee 
which is the usual variational equation \cite{Kuznetsov}. It can be used to describe how quickly a perturbation of an asymptotically stable equilibrium point will decay to zero.

\begin{defn}
\label{defn:recovery}
For $(x,y)$ in the attracting sheet $C^a$ of the critical manifold, perturbations $z = (x+u,y)$ decay to
$(x,y)$ at an exponential rate $\exp(\lambda_u)$. Note that $\lambda_u < 0$ is negative; it is called the Lyapunov exponent of $z$. The largest Lyapunov exponent has smallest magnitude and is called the leading Lyapunov exponent. If the leading Lyapunov exponent of $(x,y) \in C^a$ is $O(y^{\alpha})$ then we suggest to call $\alpha$ the recovery exponent.
\end{defn} 

The exponent $\alpha$ provides a measure how quickly perturbations in the fast direction will decay near a bifurcation depending on the distance in parameter space to the critical transition. A larger $\alpha$ indicates slower decay. The recovery exponent can easily be calculated for the four bifurcations discussed above. 

\begin{prop}
\label{eq:prop_recovery}
The recovery exponent $\alpha$ is given by 
\benn
\alpha=\left\{\begin{array}{cl}
\frac12 & \text{fold bifurcation,}\\
1 & \text{Hopf, pitchfork and transcritical bifurcation.}
\end{array}\right.
\eenn
\end{prop} 

\begin{proof}
Center manifold theory \cite{Carr} implies that it suffices to consider the vector field on the center manifold to compute the leading Lyapunov exponent for an asymptotically stable equilibrium point near a bifurcation. For the fold bifurcation we find that \eqref{eq:variational} is given by
\benn
u'=\left.\frac{\partial}{\partial X}(-Y-X^2)\right|_{X=\sqrt{-Y}}u=-2\sqrt{-Y}u.
\eenn
Therefore $H=-2\sqrt{-Y}=O(Y^{1/2})$. For the pitchfork bifurcation one gets
\benn
u'=\left.\frac{\partial}{\partial X}(YX+X^3)\right|_{X=0}u=Yu.
\eenn
The calculations for Hopf and transcritical bifurcations are equally easy. 
\end{proof}

Although Proposition \ref{eq:prop_recovery} is almost entirely obvious from a mathematical perspective it is often ignored in applications. In particular, we point out that different bifurcations can also lead to different quantitative slowing down effects. This idea is detailed for all bifurcations up to codimension two in \cite{KuehnCT2}.

\section{Fast-Slow Systems II: Dynamic Bifurcation}
\label{sec:det2} 

In the previous section we have seen that fast-slow systems provide a structural view on critical transitions. The slow change of the parameter drives the system toward a fast subsystem bifurcation at which a rapid transition occurs. We suggest to quantify the fast-slow critical transitions further.
 
\begin{defn}
\label{defn:jump}
Let $\gamma_0^c$ denote the first fast segment of a candidate trajectory starting at a critical transition point $p$. Let $\omega_c(p)$ denote the $\omega$-limit set of $\gamma_0^c$ under the fast flow. Define
\beann
l^i(p)&:=&\inf_{\gamma^c_0}\left\{d(p,\omega_c(p))\right\},\\
l^s(p)&:=&\sup_{\gamma^c_0}\left\{d(p,\omega_c(p))\right\}.
\eeann
\end{defn}

Basically $l^i(p)$ is the distance to the closest fast subsystem attractor we can jump to by starting at a critical transition while $l^s(p)$ measures the distance to the most distant attractor. In Example \eqref{ex:fold} we have $l^i(0,0)=\I=l^s(0,0)$; the same holds for normal forms of subcritical Hopf and pitchfork bifurcations. However, it is interesting to note that for the transcritical bifurcation we have $l^i=0$ and $l^s=\I$. We can use $l^{i,s}(p)$ to quantify what we described in Section \ref{sec:intro} as ``jumping to a far-away attractor''.\\

The theory of fast-slow systems provides a description of the flow near a fold critical transition for the full system with $0<\epsilon\ll1$. We briefly review this result here; see also \cite{DRvdP,KruSzm3} for further details. Consider the planar fast-slow system given by
\benn
\label{eq:fp_thm}
\begin{array}{lcl}
x' & = & -y-x^2,\\
y' & = & -\epsilon.
\end{array}
\eenn
Decompose the critical manifold as $C=C^a\cup \{(0,0)\}\cup C^r$ where 
\benn
C^a=C\cap \{x>0\}\qquad \text{and}\qquad C^r=C\cap \{x<0\}.
\eenn
For $\rho>0$ sufficiently small and a suitable interval $J\subset\R$, define a section $\Delta^{in}=\{(x,-\rho^2):x\in J\}$ transverse to $C^a$ and define a section $\Delta^{out}=\{(-\rho,y):y\in \R\}$ transverse to the fast subsystems. The next result describes the resulting flow map between the two sections \cite{KruSzm3}.

\begin{thm}
\label{thm:foldKS}
Near a generic fold bifurcation of the fast subsystem the extension of $C^a_\epsilon$ under the flow passes through $\Delta^{out}$ at a point $(-\rho,O(\epsilon^{2/3}))$. Furthermore, the transition map from $\Delta^{in}$ to $\Delta^{out}$ is a contraction with contraction rate $O(e^{-K/\epsilon})$.
\end{thm}  

One way to think of Theorem \ref{thm:foldKS} is that the trajectory of the full system does not jump at the exact fold bifurcation point but is shifted or delayed in the slow direction by $O(\epsilon^{2/3})$; see Figure \ref{fig:fig4}(b). For Hopf, pitchfork and transcritical bifurcations we also observe bifurcation delay. We shall briefly review the results for the delayed Hopf bifurcation in the simplest case; for details see \cite{Neishtadt1,Neishtadt2,Neishtadt3}. Consider a fast-slow system
\be
\label{eq:general21}
\begin{array}{lcl}
x'&=&f(x,y),\\
y'&=&\epsilon,\\
\end{array}
\ee
with $(x,y)\in\R^{2+1}$. Suppose the fast subsystem has a generic Hopf critical transition at $y=0$. Suppose for simplicity that
\benn
C=\{x_1=0=x_2\}=C^a\cup\{(0,0,0)\}\cup C^r
\eenn 
where $C^a$ is attracting for $y<0$ and repelling for $y>0$. Denote the complex conjugate pair of eigenvalues of $(D_xf)(0,0,y)$ by $\lambda_{1,2}(y)$. Consider a trajectory of the full system that enters an $O(\epsilon)$-neighborhood of $S^a$ at $y_a$ and leaves an $O(\epsilon)$-neighborhood at $y_r$; see Figure \ref{fig:fig5}. The complex phase is defined as
\benn
\Psi(\tau)=\int_{0}^\tau \lambda_1(s)ds
\eenn 
and the way-in/way-out map $\Pi$ that maps a time $\tau<0$ to a time $\Pi(\tau)>0$ is
\be
\label{eq:wayinwayout}
\text{Re}[\Psi(\tau)]=\text{Re}[\Psi(\Pi(\tau))].
\ee

\begin{figure}[htbp]
\psfrag{x1}{$x_1$}
\psfrag{y}{$y$}
\psfrag{ya}{$y_a$}
\psfrag{yr}{$y_r$}
\psfrag{C}{$C$}
\psfrag{gamma}{$\gamma$}
	\centering
		\includegraphics[width=0.75\textwidth]{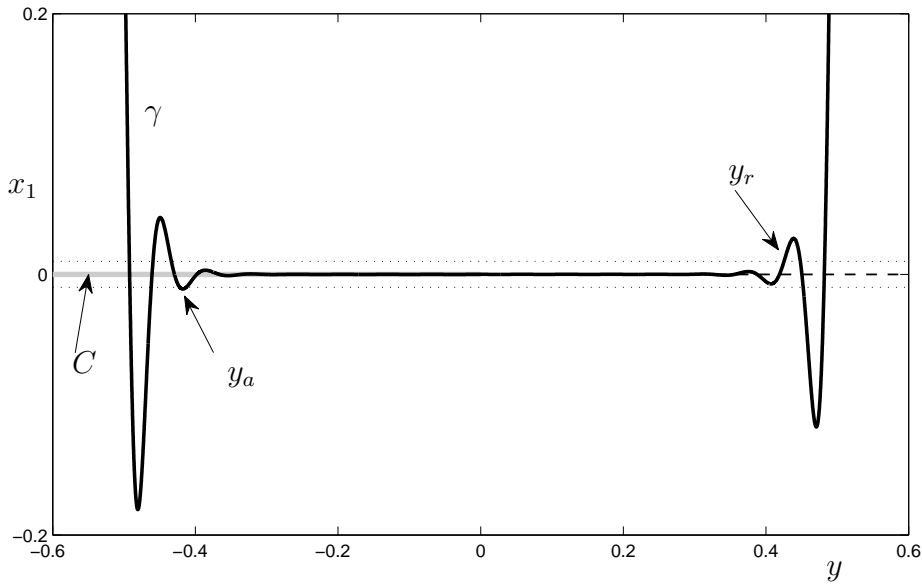}
	\caption{\label{fig:fig5}Simulation of \eqref{eq:Hopf_Euclidean} with the appended equation $y'=\epsilon$; $l_1=1$, $\epsilon=0.01$ and the initial point is $(x_1(0),x_2(0),y(0))=(0.3,0.3,-0.5)$. The trajectory $\gamma$ approaches $C$, enters a neighborhood $U=\{|x_1|<\epsilon\}$ (dotted lines) at $y_a$, follows the attracting part of the critical manifold exponentially closely, experiences a delay near a subcritical Hopf bifurcation of the fast subsystem and then leaves $U$ at $y_r$.}
\end{figure}

In principle, a few additional technical assumptions are needed on the smoothness of $f$ and the structure of the complex time level sets $\{\tau\in \C:\text{Re}[\Psi(\Pi(\tau))]=k\}\subset \C$. Unfortunately these are lengthy to state in full generality (see \cite{Neishtadt1}) but the following theorem provides the basic idea for most cases of practical interest.

\begin{thm}
\label{thm:nei}
For $0<\epsilon\ll 1$ a solution $\gamma(\tau)=\gamma(\epsilon t)$ of \eqref{eq:general21} approaching $C^a$ near $y_a$ at time $\tau_0$ will be delayed and track the unstable branch $C^r$. The map \eqref{eq:wayinwayout} can be used to approximate the delay time $\Pi(\tau_0)$ and hence to approximate $y_r\approx\Pi(\tau_0)+y_a$.
\end{thm} 

Further details about bifurcation delay and intricate special assumptions of Theorem \ref{thm:nei} can be found in \cite{Neishtadt1,BaerErneuxRinzel}. Pitchfork and transcritical transitions are treated in \cite{KruSzm4}. The delay effect shifts the critical transition in a slowly varying parameter space from $y=0$ to $y\approx y_r>0$. For $y<0$ we can predict that a critical transition is going to occur by critical slowing down; alternatively, for $y>0$ we could also observe perturbations growing exponentially. In particular, if we know $\epsilon$, $y_a$ and the type of the critical transition then we can use the theory of bifurcation delay (or dynamic bifurcation) to predict $y_r$ accurately. Let us point out that it is a key new observation that the delay effect has to be incorporated in the prediction of critical transitions and that it can potentially be useful to find early-warning signs.

\section{Stochastic Dynamical Systems}
\label{sec:stochastic_basic}

The next step is to incorporate stochastic effects to capture the role of noise in critical transitions. We start by reviewing several different viewpoints in the theory of stochastic dynamical systems. Currently the theory is less complete and structured than deterministic ODE theory. Therefore our presentation is necessarily less complete in comparison to deterministic fast-slow theory and just highlights some important ideas and methods.\\

Fix a probability space $(\Omega,\mathcal{F},\P)$ and consider a general It\^{o} stochastic differential equation (SDE)
\be
\label{eq:SDE}
dz_t=A(z_t,t)dt+B(z_t,t)dW_t
\ee
where $z\in \R^N$, $A:\R^N\times \R\ra \R^N$, $B$ is an $l\times N$-matrix and $W_t=(W_{1,t},\ldots,W_{N,t})^T$ is standard Brownian motion with components defined on $(\Omega,\mathcal{F},\P)$; we always assume that the initial conditions are deterministic and that $A$ and $B$ are sufficiently smooth maps so that existence and uniqueness results for SDEs hold \cite{Oksendal}. Alternatively we could also consider a Stratonovich SDE
\be
\label{eq:SDE_Strat}
dz_t=\tilde{A}(z_t,t)dt+\tilde{B}(z_t,t)\circ dW_t
\ee
which can, of course, be converted to an It\^{o} SDE and vice versa \cite{Gardiner}. Nevertheless, it is an important modeling question what formulation one chooses \cite{HorsthemkeLefever}. There are several complementary viewpoints to analyze \eqref{eq:SDE}-\eqref{eq:SDE_Strat} that are all envisioned to be helpful in the understanding of critical transitions.

\begin{itemize}
 \item \textit{Sample paths:} The map $\omega\mapsto z_t(\omega)$ describes a sample path for a given randomness/noise $\omega\in \Omega$. Analyzing sample paths most closely resembles the study of ODEs as one still deals with trajectories.
 \item \textit{Transition probability:} Denote the probability density of $z_t$  starting at $z_0$ at time $t_0$ by $p(z,t)=p(z,t|z_0,t_0)$ associated to \eqref{eq:SDE}. Then $p$ satisfies the forward Kolmogorov or Fokker-Planck equation \cite{Risken}
 \be
 \label{eq:FP}
 \frac{\partial }{\partial t}p(z,t)=-\sum_{j=1}^N \frac{\partial }{\partial z_j}(A_j(z,t)p(z,t))+\frac12 \sum_{j,k=1}^N \frac{\partial^2 }{\partial z_j \partial z_k}(b_{jk}(z,t)p(z,t))
 \ee
 where $b_{jk}$ are elements of the diffusion matrix $BB^T$. We shall not use the associated backward Kolmogorov equation here that is defined via the adjoint of the right-hand side of \eqref{eq:FP} in the variables $(z_0,t_0)$. Via the forward and backward Kolmogorov equations we can use the theory of parabolic partial differential equations to understand the SDE \eqref{eq:SDE}. 
 \item \textit{Random Dynamical System:} Under suitable conditions any SDE generates a random dynamical system given by a skew-product flow
 \be
 \label{eq:gen_RDS}
 (\omega,z)\mapsto (\theta(t)\omega,\varphi(t,\omega)z)=:\Theta(t)(\omega,z)
 \ee   
 on $\Omega\times \R^N$; for details of this construction see \cite{ArnoldSDE,ArnoldRDS}. The key point of \eqref{eq:gen_RDS} is that it provides a convenient framework to analyze invariant measures. Let $\mathcal{B}$ denote the Borel $\sigma$-algebra on $\R^N$. Then a measure $\mu$ on $\Omega \times \R^N$ is an invariant measure on $(\Omega\times X,\mathcal{F}\times \mathcal{B})$ if $\Theta\mu=\mu$ and $\pi_\Omega\mu=\P$ where $\pi_\Omega$ is the projection onto $\Omega$.    
\end{itemize}

Detailed introductions to some aspects of stochastic dynamics can be found in \cite{ArnoldCrauel,Araujo,Namachchivaya}. As a first step we use the sample path approach \cite{BerglundGentz,BerglundGentz1} and point out what it provides in the context of critical transitions. Let $z=(x,y)\in\R^2$ and consider the fast-slow SDE
\be
\label{eq:plane_SDE}
\begin{array}{lcl}
dx_\tau&=& \frac1\epsilon f(x_\tau,y_\tau)d\tau+\frac{\sigma_f}{\sqrt\epsilon} dW_{\tau},\\
dy_\tau&=& g(x_\tau,y_\tau)d\tau+\sigma_g dW_{\tau},\\
\end{array}
\ee
where the noise level $(\sigma_f^2+\sigma_g^2)^{1/2}=\sigma=\sigma(\epsilon)$ is usually assumed to depend on $\epsilon$; the scaling of the fast equation $\sigma_f/\sqrt\epsilon$ has been chosen due to the scaling law for Brownian motion (recall: $W_{\lambda\tau}=\lambda^{1/2}W_\tau$ in distribution for $\lambda\geq0$, \cite{Durrett2}). To understand critical transitions we would like to distinguish the region of $y$-values close to the transition from those far away. Let us first analyze the situation away from a critical transition where the deterministic critical manifold $C$ is normally hyperbolic and attracting. The deterministic slow manifold is given by
\benn
C_\epsilon=\{(x,y)\in\R^2:x=h_\epsilon(y)\}
\eenn 
where $h_\epsilon(y)=h_0(y)+O(\epsilon)$ by Fenichel's Theorem.\\ 

\textit{Remark:} Here we follow Berglund and Gentz \cite{BerglundGentz} but point out that alternative approaches for fast-slow SDEs are considered in \cite{SchmalfussSchneider} using random dynamical systems and in \cite{KabanovPergamenshchikov} using moment estimates and asymptotics.\\ 

The first goal is an estimate on the concentration of solutions to \eqref{eq:plane_SDE} near the deterministic slow manifold. To identify a neighborhood containing most sample paths we define the process
\be
\label{eq:SDE_deviate}
\xi_\tau:=x_\tau-h_\epsilon(y_\tau).
\ee  
Observe that $\xi_\tau$ measures the deviation of the fast components from the deterministic slow manifold. Applying It\^{o}'s formula to \eqref{eq:SDE_deviate} gives:
\bea
\label{eq:SDE_deviate1}
d\xi_\tau &=& dx_\tau-(D_yh_\epsilon)(y_\tau)dy+O(\sigma_g^2)d\tau \\
&=& \frac{1}{\epsilon}\left[ f(h_\epsilon(y_\tau)+\xi_\tau,y_\tau)-\epsilon (D_yh_\epsilon)(y_\tau)g(h_\epsilon(y_\tau)+\xi_\tau,y_\tau)+O(\epsilon\sigma_g^2) \right]d\tau \nonumber\\
&& +\left[\frac{\sigma_f}{\sqrt\epsilon}-\sigma_g (D_yh_\epsilon)(y_\tau)\right]dW_\tau. \nonumber
\eea 
Considering the linear approximation of \eqref{eq:SDE_deviate1} in $\xi_\tau$, neglecting the higher-order It\^{o} term $O(\epsilon\sigma_g^2)$ and replacing $y_\tau$ by its deterministic version $y_\tau^{det}$ gives
\be
\label{eq:SDE_deviate_approx}
\begin{array}{lcl}
d\xi^0_\tau&=&\frac1\epsilon A_\epsilon(y^{det}_\tau)\xi^0_\tau d\tau+\left[\frac{\sigma_f}{\sqrt\epsilon}-\sigma_g (D_yh_\epsilon)(y^{det}_\tau)\right]dW_\tau,\\
dy^{det}_\tau&=& g(h_\epsilon(y^{det}_\tau),y^{det}_\tau)d\tau,\\
\end{array}
\ee
where $A_\epsilon$ is defined as
\benn
A_\epsilon(y)= (D_xf)(h_\epsilon(y),y)-\epsilon (D_yh_\epsilon)(y)(D_xg)(h_\epsilon(y),y).
\eenn
Then define $X_\tau:=\sigma_f^{-2}\text{Var}(\xi^0_\tau)$ which satisfies a fast-slow ODE \cite{BerglundGentz} given by
\bea
\label{eq:SDE_deviate_ODE}
\begin{array}{lcl}
\epsilon \dot{X}&=& 2A_\epsilon(y)X+1, \\
\dot{y}&=& g(h_\epsilon(y),y). \\
\end{array}
\eea
The slow manifold of \eqref{eq:SDE_deviate_ODE} is
\benn
C^X_\epsilon=\left\{(X,y)\in\R^2:x=H_\epsilon(y)=-\frac{1}{2A_\epsilon(y)}+O(\epsilon)\right\}.
\eenn
The neighborhood of $C_\epsilon$ is then defined as
\be
\label{eq:nbhd_Nr}
N(r;C_\epsilon):=\left\{(x,y)\in\R^2:\frac{(x-h_\epsilon(y))^2}{H_\epsilon(y)}<r^2\right\}.
\ee
Essentially this provides a strip around $C_\epsilon$ with width depending on the variance and the linearization of the SDE; see Figure \ref{fig:fig6} for an illustration. A detailed statement and proof of the next theorem can be found in \cite{BerglundGentz}.

\begin{thm}
\label{thm:BG1}
Sample paths starting on $C_\epsilon$ stay in $N(r;C_\epsilon)$ with high probability for times approximately given by $O\left(\epsilon e^{r^2/(2\sigma_f^2)}\right)$.
\end{thm}

Theorem \ref{thm:BG1} is reminiscent of the classical Kramer's time to escape from a potential well \cite{FreidlinWentzell}. From the methodology we have just reviewed, we can obtain several important conclusions for critical transitions:

\begin{itemize}
 \item Normal hyperbolicity provides the separation criterion for sample paths into two major regimes. To develop an early-warning sign we expect to pass through a normally hyperbolic metastable regime before entering a region near the critical transition.
 \item Sample paths are likely to stay inside a neighborhood that scales with the variance. Hence if there is a critical transition due to the loss of normal hyperbolicity of a slow manifold we expect the variance to increase as we approach the transition; see also \cite{KuehnCT2}.
 \item Regarding the previous point, it is easily seen from the techniques in \cite{BerglundGentz} that this variance increase can be established rigorously for normal forms of bifurcations under suitable boundedness assumptions on the noise. In particular, there is no need to refer to heuristic arguments or autoregressive models as one can directly prove this result pathwise.
 \item The noise in the slow variable is of higher-order in the drift term of \eqref{eq:SDE_deviate1}. In the diffusion term we have noise contributions $\sigma_f/\sqrt\epsilon$ and $\sigma_g$ so that if $\sigma_f$ and $\sigma_g$ have the same asymptotic dependence on $\epsilon$ we can again neglect the slow variable noise; hence we shall only consider the case $\sigma_g=0$ from now on.
\end{itemize}

Motivated by the previous discussion and Fenichel's Normal Form Theorem \ref{thm:FNform} we are going to model paths away from a critical transition by the system 
\be
\label{eq:plane_SDE_FNF}
\begin{array}{lcl}
dx_\tau&=& \frac{\alpha}{\epsilon}(-x) d\tau+\frac{\sigma}{\sqrt\epsilon} dW_{\tau},\\
dy_\tau&=& 1d\tau,\\
\end{array}
\ee
which decouples with $y_\tau=y_0+\tau$ and where $0<\alpha=O(1)$. The fast equation of \eqref{eq:plane_SDE_FNF} is just the classical Ornstein-Uhlenbeck (OU) process \cite{Gardiner}. The solution starting at $\tau=0$ is
\be
\label{eq:OU_sol}
x_\tau=x_0e^{-\alpha \tau/\epsilon}+\frac{\sigma}{\sqrt\epsilon}\int_0^\tau e^{-\alpha(\tau-\rho)/\epsilon}dW_\rho.
\ee 
Since our initial condition is always assumed to be deterministic we get that $x_\tau$ is a Gaussian process with mean and variance given by
\beann
\E[x_\tau]&=& x_0e^{-\alpha\tau/\epsilon},\\
\text{Var}(x_\tau)&=&\left(x_0-\frac{\sigma^2}{2\alpha}\right)e^{-2\alpha\tau/\epsilon}+\frac{\sigma^2}{2\alpha}.
\eeann
The correlation is easily computed as
\benn
\E[x_\tau x_s]=\left(-\frac{\sigma^2}{2\alpha}\right)e^{-\alpha(\tau+s)/\epsilon}+\frac{\sigma^2}{2\alpha}e^{-\alpha|\tau-s|/\epsilon}.
\eenn
Observe that on a slow time scale $\tau$ of order $O(1)$ the terms involving $e^{-K\tau/\epsilon}$ are extremely small. In the limit $\tau\ra \I$
we have the stationary variance given by
\be
\label{eq:FNF_stat_var}
\lim_{\tau\ra \I}\text{Var}(x_\tau)=\frac{\sigma^2}{2\alpha}.
\ee
It is crucial to note that the variance is constant in the limit $\tau\ra \I$ but is already approximately constant up to exponentially small terms after a slow time $\tau=O(1)$. Therefore we expect that systems far away from critical transitions are characterized by a variance without a significant trend. The expectation and autocorrelation vanish in the limit $\tau\ra\I$ and also all other moments are constants.\\ 

\begin{figure}[htbp]
\psfrag{pN}{$\partial N$}
\psfrag{y}{$y$}
\psfrag{x}{$x$}
\psfrag{s2}{$\sigma^2$}
\psfrag{-s2}{$-\sigma^2$}
	\centering
		\includegraphics[width=0.95\textwidth]{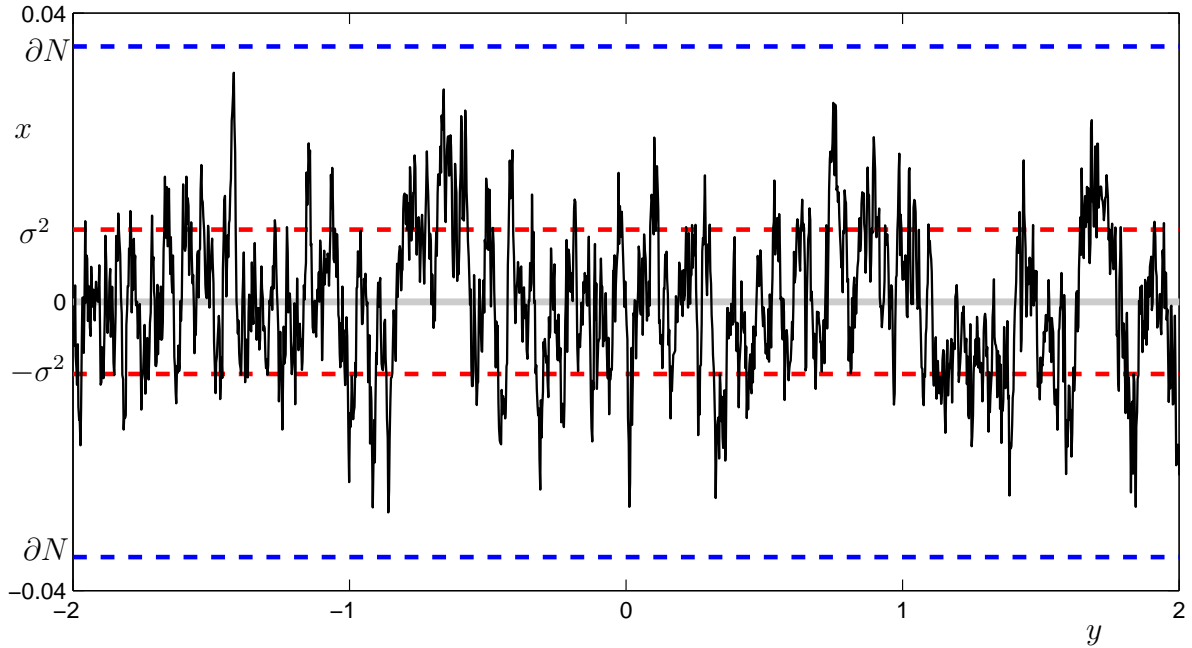}
	\caption{\label{fig:fig6}Simulation of \eqref{eq:plane_SDE_FNF} with $\epsilon=0.02$, $\sigma=0.1$ and $\alpha=1$. A sample path is shown (black) that stays inside the neighborhood $N(r;C_\epsilon)=N$ with boundaries indicated by $\partial N$ (dashed blue). We also plot a neighborhood defined by the variance $\sigma^2$ (dashed red) and the slow/critical manifold $C_\epsilon$ (gray).}
\end{figure}

Furthermore our normal form approach also suggests a way how to estimate the parameters from a single sample path. First we detrend the fast variable data in a sufficiently long normally hyperbolic phase. Using the model \eqref{eq:plane_SDE_FNF} for the fast dynamics gives
\be
\label{eq:MLE_eq}
dx_\tau= -\tilde{\alpha}x~ d\tau+\tilde{\sigma} ~dW_{\tau}
\ee
where $\tilde{\alpha}:=\alpha/\epsilon$ and $\tilde{\sigma}=\sigma/\sqrt\epsilon$. Well-known statistical techniques for parameter estimation \cite{Terrell} can then be applied to \eqref{eq:MLE_eq} to find $\tilde{\alpha}$ and $\tilde{\sigma}$ from the detrended data; for example, by using a maximum likelihood estimator or many other possible estimators \cite{Kutoyants}. This provides the correct order of magnitude for $\epsilon$ from $\tilde{\alpha}$ since $\alpha=O(1)$ by assumption. Then we can conclude the order of $\sigma$ from $\tilde{\sigma}$. This shows that the initial data far away from a critical transition can have crucial value for its prediction. Note however, that we have assumed that detrending transforms the system into Fenichel normal form. The following example shows the problems that can result in this context.

\begin{ex}
We have assumed that the fast-slow SDE \eqref{eq:plane_SDE_FNF} is already in Fenichel normal form. In general, we only have the equation for the deterministic critical manifold $C=\{(x,y)\in\R^2:f(x,y)=0\}$. We can describe $C$ as a graph $h:\R\ra \R$ so that 
\benn
C=\{(x,y)\in\R^2:x=h(y)\}
\eenn
where $f(h(y),y)=0$. Then the coordinate change $X=x-h(y)$ gives that in the new coordinates $C=\{X=0\}$. Let us consider the following example
\be
\label{eq:ex_count}
\begin{array}{lcl}
dx_t&=&(y-x)dt+\sigma dW_t,\\
dy_t&=&\epsilon g(x,y)dt.
\end{array}
\ee 
We set $(X_t,Y_t)=(x_t-y_t,y_t)$ which transforms \eqref{eq:ex_count} to
\benn
\begin{array}{lcl}
dX_t&=&-x-\epsilon g(X+Y,Y)dt+\sigma dW_t,\\
dY_t&=&\epsilon g(X+Y,Y)dt.
\end{array}
\eenn 
Then the variance of $x_t$ and $X_t$ are equal since
\benn
\text{Var}(X_t)=\text{Var}(x_t-y_t)=\text{Var}(x_t)+\text{Var}(y_t)-2\text{Cov}(x_t,y_t)=\text{Var}(x_t)
\eenn
since $y_t$ was assumed to be deterministic. In general, this cannot be assumed so that stochastic slow variables definitely will change the result; see also equation \eqref{eq:SDE_deviate1}.
\end{ex}

Hence we have identified the problem of coordinate transformation effects on critical transition indicators as a topic for future study. We are not going to consider this problem here but point out it arises immediately as a key problem once a mathematical framework for critical transitions is considered. Even without the parameter estimation problem in a normally hyperbolic regime away from the transition we must consider this problem; indeed, we might want to assume for theoretical analysis that systems are in normal form near the critical transition point.

\section{Stochastic Indicators}
\label{sec:indicators}

A natural question for finding indicators of critical transitions for SDEs is to ask what happens to the deterministic fold, Hopf, transcritical and pitchfork bifurcations under the influence of noise. This question already raises a few unanswered mathematical problems of stochastic bifurcation theory \cite{ArnoldSDE,ArnoldRPB}. We briefly review two viewpoints about what a ``stochastic bifurcation'' should be. Suppose we are given a family of random dynamical systems (RDS) $\{\Theta_y\}$ for a parameter $y\in \R$ associated to the SDE  \eqref{eq:SDE} or \eqref{eq:SDE_Strat}. Assume that $\{\mu_y\}$ is a family of invariant measures for the RDS which can be viewed as analogs for invariant sets in the deterministic case; for example, if the family of RDS has an equilibrium point at $z=0$ then $\mu_y=\delta_0$ is a natural example. We say $y=y_D$ is a dynamical or D-bifurcation point if in each neighborhood of $y_D$ there is a family of invariant measures $\nu_y$ such that $\nu_y\neq \mu_y$ and $\nu_y\ra \mu_y$ as $y\ra y_D$ in the topology of weak convergence. Basically this notion presented in \cite{ArnoldSDE} tries to capture the deterministic viewpoint of bifurcations in a stochastic context. Instead of ``qualitative changes'' for invariant measures one could also look at ``qualitative changes'' for densities associated to the SDE \eqref{eq:SDE}. Suppose $p^y_s(z)=p^y_s$ is a family of probability densities solving the stationary Fokker-Planck equation
\be
 \label{eq:FP_stat}
 0=-\sum_{j=1}^N \frac{\partial }{\partial z_j}(A_jp^y_s)+\frac12 \sum_{j,k=1}^N \frac{\partial^2 }{\partial z_j \partial z_k}(b_{jk}p^y_s).
\ee
It has been suggested to consider a qualitative change in the family of densities a $p^y_s$ a bifurcation point \cite{HorsthemkeLefever}. For example, if the density $p^y_s$ is unimodal for $y<y_P$ and bimodal for $y>y_P$ then $y=y_P$ is called a phenomenological or P-bifurcation point; substantial progress has been made to understand D- and P-bifurcations \cite{Baxendale1,Baxendale2} and associated problems of random attractors \cite{CrauelFlandoli,SchenkHoppe} and stochastic normal forms \cite{ArnoldSDE} but this theory has not yet been applied to detecting critical transitions.\\

We are going to consider an example by Arnold and Boxler \cite{ArnoldBoxler,ArnoldSDE} where explicit calculations for D- and P-bifurcations are possible; this will demonstrate that the stochastic bifurcation concepts can complement existing techniques to predict critical transitions. Consider the parametrized family of Stratonovich SDEs 
\be
\label{eq:ArnoldBoxler}
dx_t=(yx_t-x_t^2)dt+\sigma x_t\circ dW_t
\ee
representing one possible interpretation of a transcritical bifurcation with noise. Note that we could also make the parameter $y$ slowly varying; since we are working on the fast time scale $t$ this would amount to using the deterministic equation 
\benn
dy_t=\epsilon dt.
\eenn
However, the parametric analysis is already very complicated and we shall restrict to this situation here. The It\^{o} SDE associated to \eqref{eq:ArnoldBoxler} is
\be
\label{eq:ArnoldBoxlerSDE}
dx_t=\left(yx_t-x_t^2+\frac12\sigma^2 x_t\right)dt+\sigma x_t dW_t.
\ee 
Note that we are dealing with multiplicative noise with respect to the trivial solution $x_t\equiv 0$. An explicit formula \cite{ArnoldSDE} for the random dynamical system defined by \eqref{eq:ArnoldBoxler} is
\be
\label{eq:AB_formula}
\varphi_y(t,\omega)x=\frac{xe^{yt+\sigma W_t(\omega)}}{1+x\int_0^t e^{ys+\sigma W_s(\omega)}ds}.
\ee   
Ergodic invariant measures $\mu$ for RDS on $\R$ are always random Dirac measures i.e. of the form $\delta_{x_0(\omega)}$ \cite{ArnoldSDE}. From formula \eqref{eq:AB_formula} it follows that there are two families of ergodic invariant measures, one supported at $0$ given by $\mu^y_\omega=\delta_0$ and one family $\nu_\omega^y=\delta_{x^*_y(\omega)}$ supported on the random point that makes the denominator in \eqref{eq:AB_formula} zero as $t\ra \pm\I$:
\benn
x^*_y(\omega)=\left\{\begin{array}{ll} 
-\left(\int_0^\I e^{yt+\sigma W_t(\omega)}dt\right)^{-1} & \text{ for $y<0$,}\\
\left(\int^0_{-\I} e^{yt+\sigma W_t(\omega)}dt\right)^{-1} & \text{ for $y>0$.}\\
\end{array}\right.
\eenn
It is very important to note that for $y\neq 0$ the random dynamical system \eqref{eq:AB_formula} is only defined for $t>0$ on the random domain given by
\benn
D_y(t,\omega)=\left\{\begin{array}{ll} 
\left[x^*_y,\I\right) & \text{ for $y<0$,}\\
\left[0,\I\right) & \text{ for $y>0$.}\\
\end{array}
\right.
\eenn
One idea explored further in Sections \ref{sec:noiseinduced} and \ref{sec:moments} is to analyze the role of this random boundary and how it signals the explosion/critical transition of the process. Having explicit expressions for the ergodic invariant measures one can show the following bifurcation theorem \cite{ArnoldBoxler}.

\begin{thm}
The SDE \eqref{eq:ArnoldBoxler} modeling a transcritical bifurcation with multiplicative noise has a D-bifurcation at $y=0$.
\end{thm}

The D-bifurcation point provides us with an analog of the deterministic transcritical bifurcation point. We know that a critical transition is induced by a deterministic transcritical bifurcation at $y=0$. The stochastic formulation \eqref{eq:ArnoldBoxler} also provides us with additional information. Consider the stationary Fokker-Planck equation associated to \eqref{eq:ArnoldBoxler}-\eqref{eq:ArnoldBoxlerSDE}
\be
\label{eq:FP_AB}
0=-\frac{d}{dx}\left(\left([y+\frac{\sigma^2}{2}]x-x^2\right)p^y_s(x)\right)+\frac{d^2}{dx^2}\left(\frac{\sigma^2x^2}{2}p^y_s(x)\right)
\ee 
where $p^y_s(x)$ denotes the stationary probability density of $p^y(x,t)$. One normalizable solution of \eqref{eq:FP_AB} for $y>0$ is given by
\be
\label{eq:FP_sol1}
p^y_s(x)=\frac{1}{N_y}x^{\frac{2y}{\sigma^2}-1}e^{-\frac{2x}{\sigma^2}}
\ee
for $x>0$ and $p^y_s(x)=0$ for $x\leq 0$; here $N_y$ is a computable normalization constant \cite{ArnoldSDE}. From \eqref{eq:FP_sol1} we see that the density has a singularity at $x=0$ for $y\in\left(0,\sigma^2/2\right)$ and is unimodal for $y>\sigma^2/2$. Hence there is a P-bifurcation at $y_P=\sigma^2/2$; see also \cite{Zeeman1,Zeeman2} to make the non-equivalence of the two densities precise. We can either use the backward Kolmogorov equation or a symmetry argument to obtain another P-bifurcation at $y=-\sigma^2/2$ giving the final bifurcation diagram shown in Figure \ref{fig:fig7}.\\

\begin{figure}[htbp]
\psfrag{P-}{$y_{P^-}$}
\psfrag{P+}{$y_{P^+}$}
\psfrag{y}{$y$}
\psfrag{x}{$x$}
\psfrag{D}{$y_D$}
	\centering
		\includegraphics[width=0.95\textwidth]{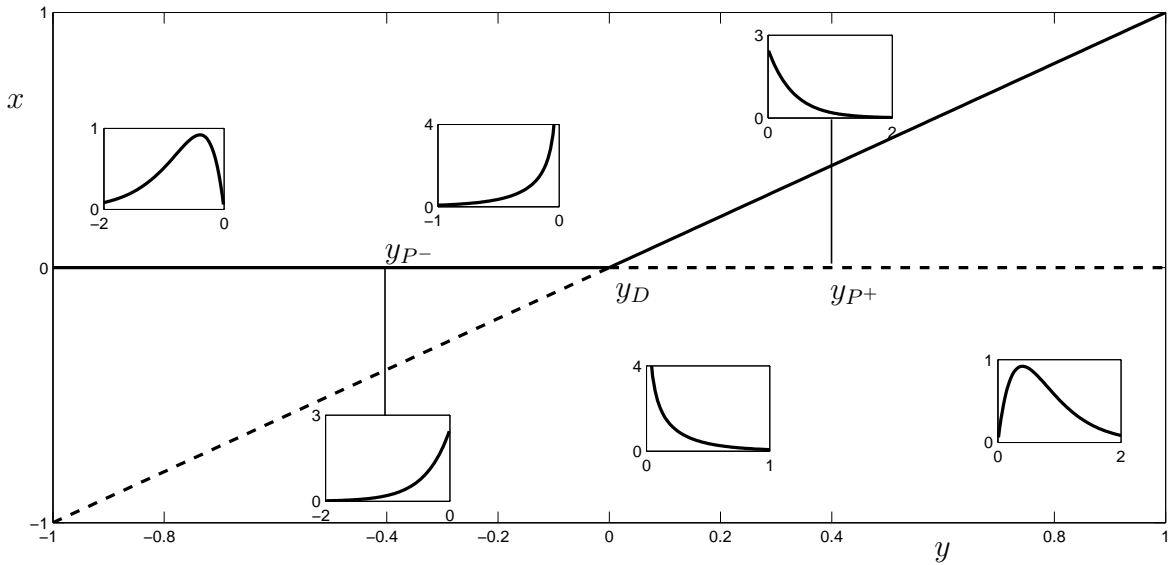}
	\caption{\label{fig:fig7}Bifurcation diagram for the Arnold and Boxler example \eqref{eq:ArnoldBoxler} with $\sigma=\sqrt{0.8}$. There are two P-bifurcations at $y_{P^\pm}=\pm\sigma^2/2=\pm0.4$ and a D-bifurcation at $y_D=0$.  The stationary densities are plotted at the values $y=\pm 0.8$ and $y=\pm0.2$ to show the qualitative change for the P-bifurcation. The deterministic transcritical bifurcation diagram is drawn for orientation purposes.}
\end{figure}

It is very interesting to calculate some of the moments of $p^y_s(x)$ explicitly; we fix $y>0$ and consider \eqref{eq:FP_sol1}. For the mean $m^s(y)$ we find
\benn
m_s^y=4^{-\frac{y}{\sigma ^2}} y \left(\frac{1}{\sigma ^2}\right)^{-\frac{2 y}{\sigma ^2}} \Gamma\left(\frac{2 y}{\sigma ^2}\right).
\eenn
The variance $v_s^y$ is
\be
\label{eq:var_AB}
v_s^y=y^2+\frac{y \sigma ^2}{2}-4^{-\frac{y}{\sigma ^2}} y \left(\frac{1}{\sigma ^2}\right)^{-\frac{2 y}{\sigma ^2}} \sigma ^2 \Gamma\left(1+\frac{2 y}{\sigma ^2}\right)+2^{-\frac{4 y}{\sigma ^2}} y^2 \left(\frac{1}{\sigma ^2}\right)^{-\frac{4 y}{\sigma ^2}} \Gamma\left(\frac{2 y}{\sigma ^2}\right)^2.
\ee

\begin{figure}[htbp]
\psfrag{V}{$v_s^y$}
\psfrag{yP}{$y_{P^+}=\frac{\sigma^2}{2}$}
\psfrag{y}{$y$}
	\centering
		\includegraphics[width=0.7\textwidth]{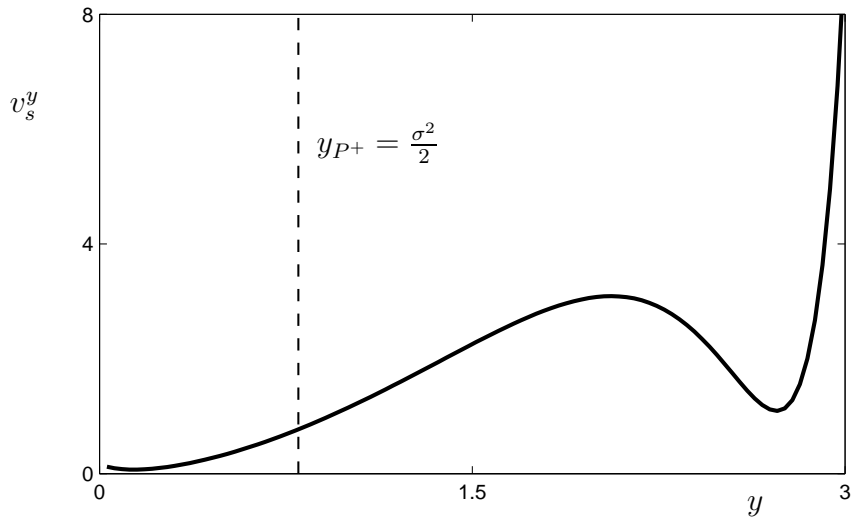}
	\caption{\label{fig:fig8}Parameter-dependent variance $v_s^y$ for the Arnold and Boxler example \eqref{eq:ArnoldBoxler}. The formula is given in equation \eqref{eq:var_AB}.}
\end{figure}

A direct plot in Figure \ref{fig:fig8} shows that the variance is non-monotone for sufficiently small noise $\sigma$. In particular, there is a local minimum and a local maximum for $y>y_P$. By symmetry this situation also holds for $y<y_P$. There are several observations that we can conclude from the previous discussion regarding critical transitions:

\begin{itemize}
 \item There is a P-bifurcation preceding a D-bifurcation for the transcritical bifurcation occurring in \eqref{eq:ArnoldBoxler}. In particular, the P-bifurcation point can potentially be used as an estimator/predictor for the critical transition point.
 \item The D-bifurcation point could be used to form the ``organizing center'' for the critical transition in analogy to the bifurcation point in the deterministic case i.e. it provides us with a rigorous definition of a reference point where the jumps occur.
 \item The unstable deterministic equilibrium branches naturally appear as boundary points for the stationary Fokker-Planck equation.
 \item The variance, and also other moments, can vary rapidly and non-monotonically near a critical transition point; cf. the situation in Section \ref{sec:moments}.
 \item For the non-stationary case, the boundaries for the dynamical system have to be random since there is always a positive probability that a sample path reaches any positive or negative $x$-value. We shall discuss this problem in Section \ref{sec:noiseinduced}.
\end{itemize}

We remark that the example by Arnold and Boxler is rather special since we were able to find explicit solutions for all interesting quantities. In many cases we would have to rely more on numerical methods; see, for example, \cite{MilsteinTretyakov,KloedenPlaten}. Furthermore, it has been shown that D- and P-bifurcations do not always have to appear together and that the situation for Hopf bifurcation is much more complicated than anticipated \cite{ArnoldRPB}. However, examples with multiplicative noise are expected to appear naturally in many control problems since approaching an instability  also might want to reduce the noise level. In this scenario it is easy to understand that for multiplicative noise a rising variance early-warning sign can fail \cite{KuehnCT2}. Therefore we suggest that P-bifurcation indicators should definitely be added to the toolbox of possible early-warning signs.

\section{Noise-Induced Transitions}
\label{sec:noiseinduced}

The term ``noise-induced transitions'' groups together a rather wide spectrum of phenomena; other terms that are related to it are stochastic resonance, coherence resonance, self-induced stochastic resonance \cite{Lindneretal}. The different concepts share a common feature: the noise induces dynamical behavior in a system that cannot be found in the deterministic version. To illustrate the situation consider the following planar fast-slow SDE
\be
\label{eq:fold_nit}
\begin{array}{lcl}
dx_\tau&=& \frac1\epsilon(y-x^2) d\tau+\frac{\sigma}{\sqrt\epsilon} dW_{\tau},\\
dy_\tau&=& g(x_\tau,y_\tau)d\tau,\\
\end{array}
\ee
modeling the fold critical transition. If we consider \eqref{eq:fold_nit} on the fast time scale $t=\tau/\epsilon$ and then consider the singular limit $\epsilon\ra 0$ we get
\be
\label{eq:WF_ex}
dx_t=(y-x^2)dt+\sigma dW_t.
\ee
Fixing some $y>0$ a sample path starting for some $x\approx \sqrt{y}$ is expected to stay with high probability near the stable equilibrium of the deterministic system at $x=\sqrt{y}$ if $\sigma$ is sufficiently small; see Theorem \ref{thm:BG1}. The problem is that it can escape from a neighborhood of \eqref{eq:WF_ex} eventually with some probability i.e. there is a large deviation. Classical theory of large deviations \cite{FreidlinWentzell} predicts how likely it is to escape from an attracting equilibrium. The deterministic version of \eqref{eq:WF_ex} is a gradient system with potential
\benn
U(x)=-yx+\frac13x^3.
\eenn 
The potential difference to go from the stable equilibrium $x=\sqrt{y}$ past the unstable equilibrium at $x=-\sqrt{y}$ is
\benn
H:=U(\sqrt{y})-U(-\sqrt{y})=\frac43 y^{3/2}.
\eenn   
Then it is a classical result in large deviations \cite{FreidlinWentzell,BerglundGentz} that it takes a time $t=O(e^{2H/\sigma^2})$ for an excursion past the unstable equilibrium to occur. If $y=O(1)$ and $0<\sigma\ll 1$ then these excursions are extremely rare and one expects that the fast-slow system \eqref{eq:fold_nit} behaves deterministically and that Theorem \eqref{thm:foldKS} applies to analyze the critical transition. The key point for this line of reasoning is that we have assumed that
\be
\label{eq:small_fs_noise}
0<\sigma\ll \sqrt{\epsilon}\ll1
\ee     
for equation \eqref{eq:WF_ex} i.e. that the noise is small with respect to the time scale separation. In fact, one can show that excursions are very likely if the roles in \eqref{eq:small_fs_noise} are reversed \cite{BerglundGentz}.

\begin{thm}
\label{thm:BG2}
Consider the SDE \eqref{eq:fold_nit} and suppose $g\equiv 1$. If $\sigma\ll \sqrt{\epsilon}$ then critical transitions before the deterministic fold bifurcation point occur with very small probability. For $\sigma\gg \sqrt{\epsilon}$ critical transitions before the deterministic fold bifurcation occur with very high probability.
\end{thm}

The detailed estimates and the derivation of the scaling law can be found in \cite{BerglundGentz}. Theorem \ref{thm:BG2} confirms our intuition that noise larger than the time scale separation can make the system jump away from an attracting critical manifold and that a fast-slow system with very small noise should closely resemble the deterministic situation. We also say that  
\benn
\sigma\approx \sqrt{\epsilon}
\eenn 
marks the intermediate regime. Similar results should also hold for transcritical and pitchfork bifurcations but with a different scaling law. The situation is less studied but the results in \cite{BerglundGentz} indicate that 
\be
\label{eq:BG3}
\sigma\approx \epsilon^{3/4}
\ee
is the intermediate regimes for the transcritical and pitchfork bifurcations. An additional problem arises when the slow variables representing the parameters have non-trivial slow dynamics. Consider the following stochastic van der Pol equation (see also \cite{Lindneretal}):
\be
\label{eq:SDE_vdP}
\begin{array}{lcl}
dx_\tau &=& \frac1\epsilon \left(y_\tau-\frac{x_\tau^3}{3}+x_\tau\right)d\tau+\frac{\sigma}{\sqrt\epsilon} dW_\tau,\\
dy_\tau &=& (a-x_\tau)d\tau.\\
\end{array}
\ee
For $a>1$ the deterministic equation has a unique globally stable equilibrium at $x=a$. The deterministic critical manifold is 
\benn
C=\left\{(x,y)\in\R^2:y=\frac{x^3}{3}-x\right\}.
\eenn
It is normally hyperbolic away from the two fold points $x=\pm 1$ and naturally splits into three parts
\benn
C^{a,-}=C\cap \{x<-1\},\qquad C^r=C\cap \{-1<x<1\},\qquad C^{a,+}=\{x>1\}
\eenn
where $C^{a,\pm}$ are attracting and $C^r$ is repelling. In Figure \ref{fig:fig1} a direct numerical simulation using the Euler-Maruyama method for SDEs \cite{Higham} is shown.\\

\begin{figure}[htbp]
\psfrag{x}{$x$}
\psfrag{y}{$y$}
	\centering
		\includegraphics[width=0.8\textwidth]{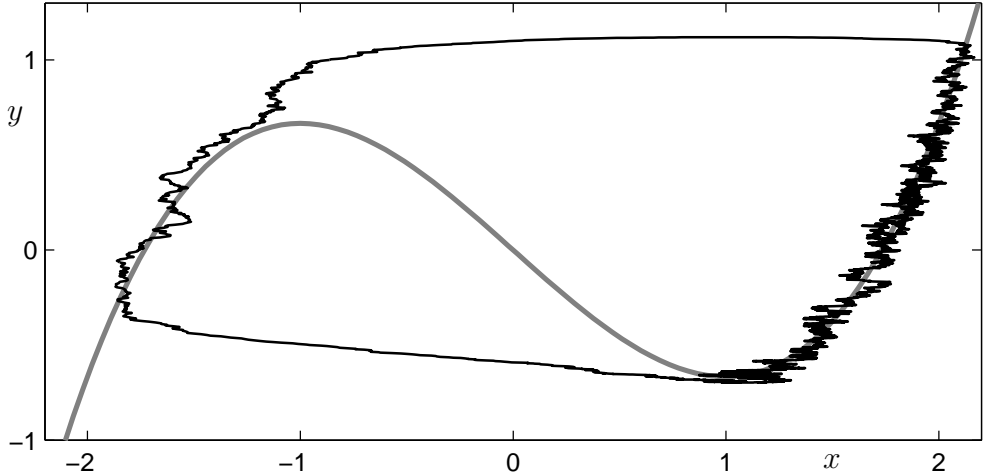}
	\caption{\label{fig:fig1}Single sample path (black) for equation \eqref{eq:SDE_vdP} with parameter values $(\epsilon,a,\sigma)=(0.05,1.05,0.1)$. The critical manifold $C$ (grey) also shown. The path was started at $(x(0),y(0))=(2,2/3)$ and has been stopped at $\tau=2400$.}	
\end{figure}

Observe from Figure \ref{fig:fig1} that the sample path is not even close to the deterministic solution which converges to the deterministic equilibrium at $x=1.05$. A noise induced transition has occurred near the deterministic equilibrium point close to the fold point at $x=1$. This transition induced a sample path that resembles classical relaxation oscillations; for an asymptotic analysis of scaling laws in the double limit $(\epsilon,\sigma)\ra(0,0)$ we refer to \cite{MuratovVanden-Eijnden,MuratovVanden-EijndenE}. From the discussion in this section we can conclude the following for critical transitions:

\begin{itemize}
 \item Critical transitions are expected to occur before reaching the neighborhood of a deterministic bifurcation point if the noise level is larger than the time scale separation.
 \item If the noise is small compared to the time scale separation (e.g. $\sigma\ll \sqrt\epsilon$ in the fold transition) we expect the deterministic bifurcation point to be a good prediction for the location of the critical transition.
 \item Scaling laws between noise and time scale separation will play a crucial role whether critical transitions are predictable at all and what phenomena can occur as we approach a transition \cite{BerglundGentzKuehn}.
 \item A slow variable/parameter with non-trivial dynamics can cause very complicated noise-induced transitions if $g$ is not bounded away from zero near the bifurcation point. The situation is even more complicated once multiple slow variables are considered \cite{KuehnCT2,BerglundGentzKuehn}.
\end{itemize}   

\section{Variance I: Analysis}
\label{sec:moments}

In this section we calculate the variance before a critical transition for several bifurcations in the singular limit. We consider the fast-slow SDE
\be
\label{eq:fold_simple_calc}
\begin{array}{lcl}
dx_t&=& f(x_t,y_t) dt+\sigma dW_t,\\
dy_t&=& \epsilon dt,\\
\end{array}
\ee
for $(x,y)\in\R^2$ and $\sigma>0$ is constant. The function $f(x,y)$ will be the vector field for the normal forms of the fold, transcritical and pitchfork bifurcations. Since we are only interested in the moments before the transition, we consider the normal forms only for $y<0$ as given in Section \ref{sec:det1}. In the singular limit $\epsilon\ra0$, the fast subsystem is one-dimensional with transition probabilities $p^y(x,t)=p^y(x,t|x_0,t_0)$ satisfying the Fokker-Planck equation
\be
\label{eq:FP_simple1}
\frac{\partial }{\partial t}p^y(x,t)=-\frac{\partial}{\partial x}(f(x,y)p^y(x,t))+\frac{\sigma^2}{2}\frac{\partial^2}{\partial x^2}p^y(x,t)
\ee
posed on some interval $(a,b)\subset \R$ with initial condition $p^y(x,t_0|x_0,t_0)=\delta(x-x_0)$. The probability current $J$ is defined by
\benn
J(x,t)=f(x,t)p^y(x,t)-\frac{\sigma^2}{2}\frac{\partial}{\partial x}p^y(x,t).
\eenn
Let us assume that there is a stationary distribution $p^y_s=p^y_s(x)$ for the process then \eqref{eq:FP_simple1} reduces to
\be
\label{eq:FP_simple2}
\frac{\partial}{\partial x}(f(x,y)p^y_s(x))-\frac{\sigma^2}{2}\frac{\partial^2}{\partial x^2}p^y_s(x)=0
\ee
which means that $J=J(x)$ satisfies $J'(x)=0$ and hence $J(x)=\text{constant}$; if we add reflecting boundary conditions then $J=0$ and it follows that 
\be
\label{eq:FP_simple3}
f(x,y)p^y_s(x)-\frac{\sigma^2}{2}\frac{\partial}{\partial x}p^y_s(x)=0.
\ee
The last equation can be integrated directly to give the classical potential solution
\benn
p^y_s(x)=\frac{1}{\mathcal{N}}\exp\left(2\int_a^x \frac{f(w,y)}{\sigma^2}dw\right)
\eenn
where $\mathcal{N}$ is the normalization constant for the probability distribution $\mathcal{N}=\int_a^b p^y_s(x)ds$. For each of the normal forms we choose the boundary points as follows:
\be
\label{eq:reflec_boundaries}
\begin{array}{lll}
\text{fold} & f=f_1(x,y)=-y-x^2 & (a,b)=(-\sqrt{-y},\I), \\
\text{transcritical }\quad & f=f_2(x,y)=yx-x^2 \qquad & (a,b)=(y,\I),\\
\text{pitchfork} & f=f_3(x,y)=yx+x^3 & (a,b)=(-\sqrt{-y},\sqrt{-y}).\\
\end{array}
\ee

The choices are motivated by two factors. In Section \ref{sec:indicators} we observed that the random dynamical system induced by the SDE \eqref{eq:fold_simple_calc} is described by limiting its domain to points which do not escape. We eliminate the random boundaries and consider the unstable equilibria (i.e. the repelling parts of the critical manifold) as boundaries. Furthermore, our choice of reflecting boundaries enforces the condition that transitions only occur after the deterministic critical transition. We get the following stationary densities
\be
\label{eq:res_dist}
\begin{array}{ll}
\text{fold} & p^y_{s,1}(x)=\frac{1}{\mathcal{N}_1}\exp\left(\frac{2}{\sigma^2}\left[ -yx-\frac13x^3+\frac23(-y)^{3/2}\right]\right), \\
\text{transcritical }\quad & p^y_{s,2}(x)=\frac{1}{\mathcal{N}_2}\exp\left(\frac{2}{\sigma^2} \left[ \frac12 yx^2-\frac13x^3-\frac16 y^3\right]\right), \\
\text{pitchfork} & p^y_{s,3}(x)=\frac{1}{\mathcal{N}_3}\exp\left(\frac{2}{\sigma^2}\left[ \frac12 yx^2+\frac14 x^4+\frac14 y^2 \right]\right). \\
\end{array}
\ee
By comparing \eqref{eq:res_dist} to the Gaussian density of \eqref{eq:OU_sol}, we observe a transition from symmetric to asymmetric behavior for the fold and transcritical transitions. However, the density for the pitchfork transition is still $\Z_2$-symmetric with respect to $x\mapsto -x$. Furthermore there are no P-bifurcations for any $p^y_{s,j}(x)$ for $y<0$ and $j=1,2,3$. This shows that symmetry-breaking and P-bifurcations are not necessarily early-warning signs of critical transitions.\\

\begin{figure}[htbp]
\psfrag{Var}{$Var$}
\psfrag{y}{$y$}
\psfrag{pitchfork}{pitchfork}
\psfrag{transcritical}{transcritical}
\psfrag{fold}{fold}
	\centering
		\includegraphics[width=0.85\textwidth]{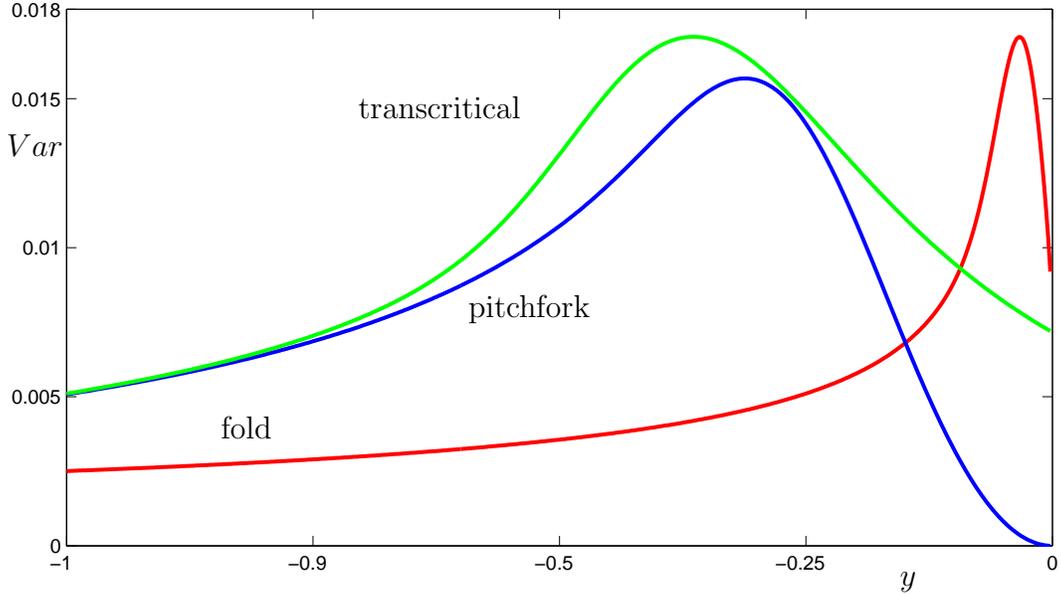}
	\caption{\label{fig:fig2}Variances $Var$ for \eqref{eq:res_dist} depending on the parameter $y$ with $\epsilon=0$; transcritical (green), pitchfork (blue) and fold (red)  transitions are considered. Starting from $y\ll -1$ the variance is almost constant, then we see that for all three cases there is a clearly visible rapid increase in the variance as the deterministic critical transition is approached. However, due to the reflecting boundary conditions we have chosen for the singular limit Fokker-Planck equation, the variance decreases again near $y=0$.}	
\end{figure} 

Figure \ref{fig:fig2} shows the variance of each distribution as a function of the parameter $y$ for a given fixed noise $\sigma=0.1$. Starting the parameter from $y\ll -1$ and increasing it, we see that for all critical transitions there is a rapid increase in the variance as the deterministic critical transition is approached. This confirms the observations and predictions from Section \ref{sec:intro} for our normal form SDE models. However, we also observe that there are local maxima for each curve as we increase $y$ further. The local maxima are caused by our modeling approach using the reflecting boundaries; the density becomes more and more confined near the stable critical manifold as we approach $y=0$. Note that this does not contradict results using a sample paths approach as for sample paths the scaling of the variance is calculated without boundaries at unstable equilibrium points and for $\epsilon>0$. Another interesting conjecture from Figure \ref{fig:fig2} is that the additional local maxima that we have obtained using reflecting boundaries can be viewed as the locations where a linearized approximation fails. More precisely, when $y\ll -1$ then we are in a normally hyperbolic regime and linearization and results about OU-process are applicable. When we get closer to the critical transition, nonlinear effects and noise-induced phenomena have to be taken into account. Furthermore it is easily calculated from the formulas \eqref{eq:res_dist} that the local maxima of the variance move closer to the critical transition if we decrease the noise level. This shows that by choosing boundary conditions for the Fokker-Planck equation we not only guarantee the existence of a normalizable density in the singular limit but also obtain additional information about critical transitions by identifying an easily-to-calculate indicator (``the local maximum'') beyond which linearized theory definitely fails. This shows that a dynamic sample paths viewpoint $(\epsilon>0)$ and a singular limit (or quasi-static, $\epsilon=0$) approach to critical transitions can nicely complement each other.\\ 

From Figure \ref{fig:fig2} we can also conclude that the variance curves for the transcritical/pitchfork transition are substantially different from the fold transition. Since the two cases also have different recovery exponents for slowing down (see Proposition \ref{eq:prop_recovery}) it should be possible to distinguish between them using early warning signs. 
   
\section{Variance II: Numerical Simulation}
\label{sec:simulation}

To relate our results in Section \ref{sec:moments} more directly to techniques used in applications we consider numerical simulation of sample paths \cite{Higham,KloedenPlaten,MilsteinTretyakov}. As a first question we address what happens to the variance for $0<\epsilon\ll 1$ in comparison to the singular limit calculation Fokker-Planck calculation.\\ 

\begin{figure}[htbp]
\psfrag{a}{(a1)}
\psfrag{a1}{(a2)}
\psfrag{b}{(b1)}
\psfrag{b1}{(b2)}
\psfrag{c}{(c1)}
\psfrag{c1}{(c2)}
\psfrag{Var}{$Var$}
\psfrag{y}{$y$}
\psfrag{escape}{\% esc.}
	\centering
		\includegraphics[width=0.9\textwidth]{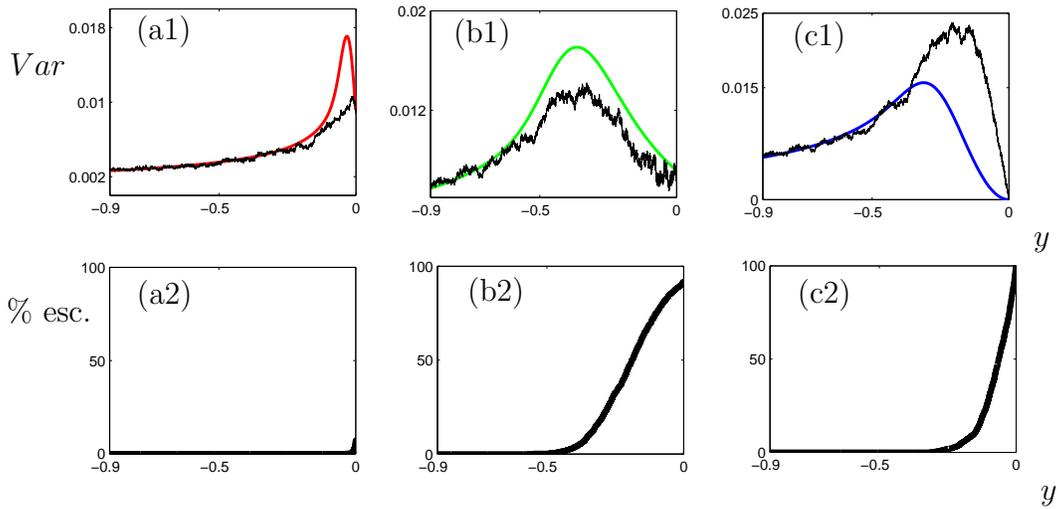}
	\caption{\label{fig:fig9}Variances $Var$ depending on the parameter $y$; transcritical (green), pitchfork (blue) and fold (red) transitions are taken from Figure \ref{fig:fig2} with $\sigma=0.1$. The black curves have been computed from 1000 sample paths with $(\sigma,\epsilon)=(0.1,0.02)$. A path beyond the unstable critical manifold at some $y=y_c$ (see boundaries in equation \eqref{eq:reflec_boundaries}) is counted as an escaped path and is not considered for the variance with $y>y_c$; note that the colored curves from Figure \ref{fig:fig2} have been computed with reflecting boundaries and $\epsilon=0$. The figures (a1),(b1),(c1) show the variance and (a2),(b2),(c2) the percentage of escaped trajectories for the fold, transcritical and pitchfork transitions respectively.}	
\end{figure} 

Again we consider the fast-slow SDE \eqref{eq:fold_simple_calc} for the fold, transcritical and pitchfork normal forms given in \ref{eq:reflec_boundaries}. Figure \ref{fig:fig9} shows the variance of the $x$-variable, for each value of $y$, calculated from 1000 sample paths. More precisely, if we index the sample paths by $j=1,2,\ldots,1000$ we compute the variance of the fast variable $\{x^j_t\}_j$ for a fixed time $t$; since $y=\epsilon t$, we expect to re-compute an approximation to the variance for the stationary distributions $p^y_s(x)$ if $\epsilon$ is sufficiently small, as long as we are not too close to $y=0$ where noise-induced transitions and reflecting boundary effects are dominant. We have fixed the parameter values to $(\sigma,\epsilon)=(0.1,0.02)$ which means for the fold bifurcation we rarely expect noise-induced transitions. Due to the different scaling laws for the transcritical and pitchfork bifurcations, we do expect noise-induced transitions in this case; cf. \cite{BerglundGentz} and equation \eqref{eq:BG3}. The percentage of escaped trajectories is shown in Figure \ref{fig:fig9}(a2),(b2),(c2). The computed variance of the sample paths is shown in Figure \ref{fig:fig9}(a1),(b1),(c1) as black curves.\\

\begin{figure}[htbp]
\psfrag{a}{(a)}
\psfrag{b}{(b)}
\psfrag{c}{(c)}
\psfrag{Var}{$Var$}
\psfrag{y}{$y$}
	\centering
		\includegraphics[width=0.9\textwidth]{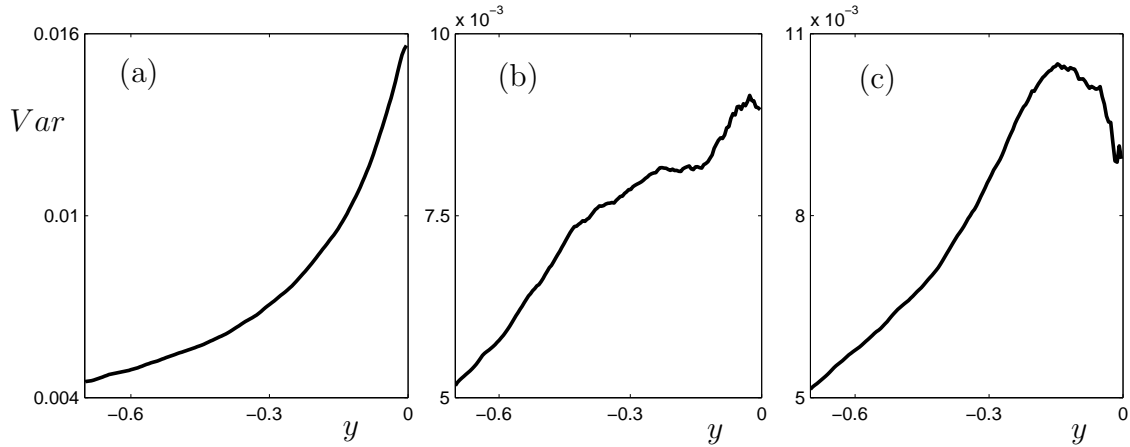}
	\caption{\label{fig:fig10} Sample variances $Var$ from \eqref{eq:var_Scheffer} depending on the parameter $y$; (a) fold bifurcation, (b) transcritical bifurcation and (c) pitchfork bifurcation. The black curves have been computed from 1000 sample paths with $(\sigma,\epsilon)=(0.1,0.02)$. A path beyond the unstable critical manifold at some $y=y_c$ (see boundaries in equation \eqref{eq:reflec_boundaries}) is counted as an escaped path and not considered for the variance with $y>y_c$.}	
\end{figure} 

Note that our initial prediction of variance increase from Section \ref{sec:moments} is correct but our simple stationary distribution method fails to capture the results correctly very close to the transition point. This is expected as sample paths are counted as escaped path for the numerical simulation once they reach the boundaries defined in \eqref{eq:reflec_boundaries} (``absorbing boundaries'', $\epsilon>0$) whereas the Fokker-Planck calculation in Section \ref{sec:moments} assumed reflecting boundary conditions and $\epsilon=0$. This shows that due to the reflecting boundaries the variance is decreased near the transition point. The interesting conclusion from Figure \ref{fig:fig10} is that different modeling techniques can produce different estimates for the moments in critical transition normal forms near the transition point. As long as we are far enough away in our approach the theories match up. This suggests to focus on this initial regime away from the bifurcation; this analysis is carried out in detail for all bifurcations up to codimension two in \cite{KuehnCT2}.\\

\begin{figure}[htbp]
\psfrag{x}{$x$}
\psfrag{y}{$y$}
	\centering
		\includegraphics[width=0.85\textwidth]{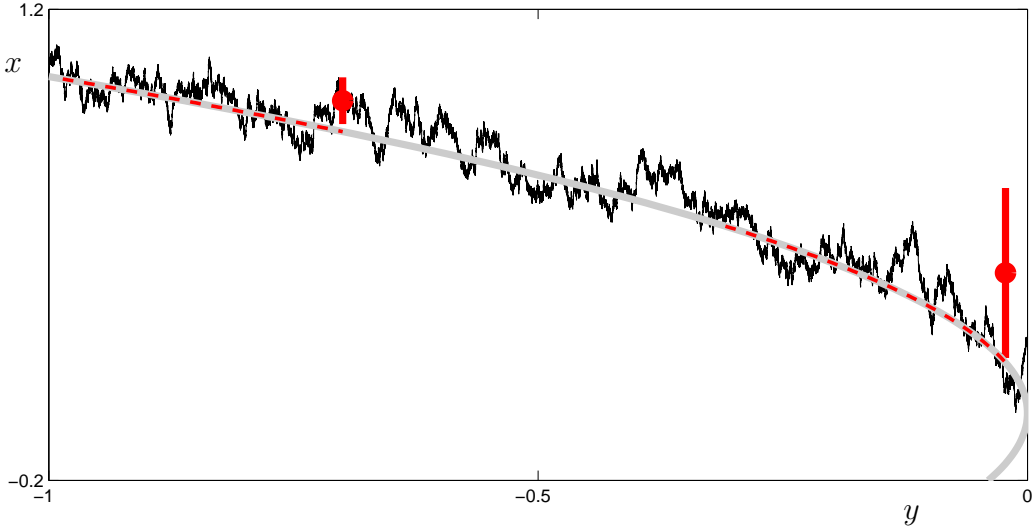}
	\caption{\label{fig:fig11} Sample path near a fold critical transition (black); parameters are $(\sigma,\epsilon)=(0.1,0.02)$. The deterministic critical manifold $C$ is shown in grey and two subsets are marked (dashed red) which correspond to windows of length $y\approx 0.2861$. From these two windows we compute two sample variances $V(t^*_{1,2})=V_{1,2}$ where $\epsilon t_1\approx -0.7$ and $\epsilon t^*_2\approx -0.02$ according to \eqref{eq:var_Scheffer}. The mean values $\mu(t^*_{1,2})=\mu_{1,2}$ are marked with red dots. The variance is indicated by a red vertical lines $[\mu_{j}-V_j,\mu_{j}+V_j]$ for $j=1,2$ that have been centered at the mean values and stretched by a factor of $20$ to make the variances visible.}	
\end{figure} 

However, a major problem arises in a practical context, if we only have a single sample path to predict a critical transition, say $\gamma_t=(x_t,y_t)$ for $t\in[0,T]$. Usually one computes an early-warning sign by considering a finite time interval (or window) of length $s<T$ and computes the sample path variance for this time interval \cite{Schefferetal}. Suppose $\gamma_t$ is known on a grid of times $t_j$ with $t_0=0$ and $t_{N-1}=T$ so that $N^*$ time points fall into an interval of time length $s$. Then the sample mean for the fast variable $x$ for some $t^*\in[s,T]$ is
\benn
\mu(t^*):=\mu([t^*-s,t^*])=\frac{1}{N^*}\sum_{t_j\in[t^*-s,t^*]} x_{t_j}
\eenn   
and the sample variance is
\be
\label{eq:var_Scheffer}
V(t^*):=Var([t^*-s,t^*])=\frac{1}{N^*}\sum_{t_j\in[t^*-s,t^*]} \left\{x_{t_j}-\mu([t^*-s,t^*])\right\}^2.
\ee   
Figure \ref{fig:fig10} shows the sample variance for $(\sigma,\epsilon)=(0.1,0.02)$. A window of size $s\approx 14.3051$ is used which corresponds an interval of length $\approx 0.2861$ for $y$ as $y=\epsilon t=0.02 t$. For the transcritical and pitchfork bifurcations in Figure \ref{fig:fig10}(b)-(c) we obtain shifted versions of the stationary variances i.e. the variance increases but local maxima are moved towards the critical transition. This is expected since the sample variance ``lags behind'' the stationary estimator that is computed at a fixed $y$ for $0\leq\epsilon\ll1$.\\

The sample variance indicator for the fold transition in Figure \ref{fig:fig10}(a) shows a clear monotone increasing deterministic trend and does not seem to lag behind the stationary variance calculation/simulation. This can be explained easily from the fast-slow geometry of the SDE as follows. Consider a single sample path near the fold transition shown in Figure \ref{fig:fig11} at parameter values $(\sigma,\epsilon)=(0.1,0.02)$. In Figure \ref{fig:fig11} two subsets of the deterministic critical manifold are marked (dashed red) which correspond to windows of length $y\approx 0.2861$. From these two windows we compute two sample variances $V(t^*_{1,2})=V_{1,2}$ where $\epsilon t_1\approx -0.7$ and $\epsilon t^*_2\approx -0.02$ according to \eqref{eq:var_Scheffer}. The mean values $\mu(t^*_{1,2})=\mu_{1,2}$ are marked with red dots. The variance is indicated by a red vertical lines $[\mu_{j}-V_j,\mu_{j}+V_j]$ for $j=1,2$ that have been centered at the mean values and stretched by a factor of $20$ to make the variances easier to visualize. It is now obvious why the variance must increase ``deterministically'' near fold critical transition if measured using \eqref{eq:var_Scheffer}; the critical manifold is locally parabolic and has much higher curvature near $y=0$. Since the window size for the measurement has to be rather large to measure anything meaningful, the sample mean $\mu_2$ is located further away from the critical manifold. Hence the sample variance will be larger due to geometric considerations and without even considering the noise effect. A good way to think about the situation is to project the subsets of the sample path corresponding to the two measurement windows onto the vertical red lines in Figure \ref{fig:fig11}. The same argument does not hold for the transcritical and pitchfork bifurcations as the stable critical manifold before the transition is given by $x=0$. This shows that practical measurement techniques have to be applied and interpreted very carefully if only a single sample path is available. 

\section{Autocorrelation}
\label{sec:autocorr}

\begin{figure}[tbp]
\psfrag{y}{$y$}
\psfrag{R}{$R$}
	\centering
		\includegraphics[width=0.8\textwidth]{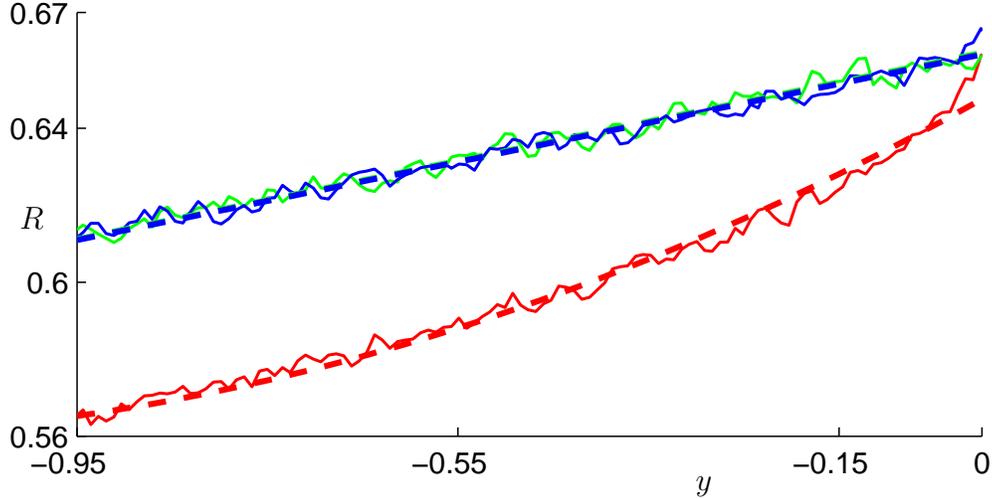}
	\caption{\label{fig:fig13}Plot of the lag-$k$ autocorrelation with $k=0.002$ for $10000$ sample paths for each of the three one-dimensional critical transitions: fold (red), transcritical (green) and pitchfork (blue). The solid thin lines are numerical data and the thick dashed lines are approximations (quadratic for the fold and linear for transcritical/pitchfork). Parameter values for the simulation are $\sigma=0.1$ and $\epsilon=0.02$.}	
\end{figure} 

\begin{figure}[htbp]
\psfrag{y}{$y$}
\psfrag{R}{$R$}
	\centering
		\includegraphics[width=0.8\textwidth]{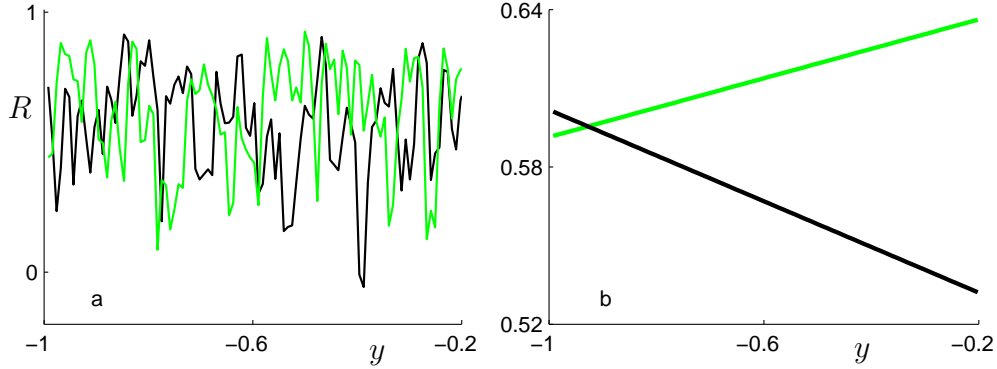}
	\caption{\label{fig:fig14}Plot of the lag-$k$ autocorrelation with $k=0.002$ for $2$ sample paths for the transcritical transition. (a) The thin lines are numerical data. (b) The thick lines are linear approximations of the curves in (a). Parameter values for the simulation are $\sigma=0.1$ and $\epsilon=0.02$.}	
\end{figure} 

Increasing autocorrelation has been proposed as an early warning sign for a critical transition \cite{HeldKleinen,Dakosetal,Schefferetal}. As a first step we calculate the autocorrelation from numerical simulation averaged over $10000$ sample paths for the normals forms of the fold, transcritical and pitchfork transitions; see equations \eqref{eq:fold_simple_calc} and \eqref{eq:reflec_boundaries}. The lag-$k$ autocorrelation can be estimated from a time series $(x_1,x_2,\ldots,x_n)$ by the formula
\benn
R(k):=\frac{1}{(n-k)v^2}\sum_{l=1}^{n-k}  (x_l-\mu)(x_{l+k}-\mu)
\eenn
where $\mu$ and $v$ are the sample mean and variance. We counted a sample path as an escaped path once it leaves the set $\{(x,y)\in\R^2:x>-1\}$ for the fold and transcritical transitions; for the pitchfork transition we consider sample paths only inside the set $\{(x,y)\in\R^2:|x|<1\}$. Figure \ref{fig:fig13} shows the results for the lag-$k$ autocorrelation with a short lag of $k=0.002$ computed from a subsegment of the sample path of length $8k$. We do not discuss the different choices regarding the lag $k$ or the choice of time series subsegments but remark that practical applications might have to deal with short time series data. There is a visible increasing trend in the autocorrelation for all three critical transition point. The autocorrelation for the transcritical and pitchfork transitions increases linearly and the two cases are virtually indistinguishable by this measure. The fold autocorrelation seems to increase quadratically. This shows that the increase in autocorrelation can be found in our SDE normal forms as an indicator for a critical transition.\\

As for the variance, it is more problematic to interpret the autocorrelation as an indicator for a single sample path. The problem is demonstrated in Figure \ref{fig:fig14} for two sample path approaching the transcritical transition. The autocorrelation fluctuates rapidly as $y$ slowly increases; see Figure \ref{fig:fig14}(a). As a first approach to check whether it is increasing or decreasing we consider a linear approximation as in Figure \ref{fig:fig13}. These lines are shown in Figure \ref{fig:fig14}(b) and one increases (green) while the other decreases (black). We know that on average we expect an increasing autocorrelation but we would make an incorrect prediction from the black sample path. This demonstrates a need for a detailed analysis of the dependence of different indicators on the parameters. For example, for the autocorrelation we have the system parameters $(\epsilon,\sigma)$ and the measurement parameters $(k,n)$ for the lag-$k$ autocorrelation of a time series of length $n$.  
 
\section{Discussion}
\label{sec:discussion}

In this paper we have given an overview of the mathematical tools that can be applied to critical transitions. Our main viewpoint is that studying normal-form fast-slow stochastic dynamical systems should provide an additional route to understand critical transitions beyond studying models arising directly from applications. Standard methods from fast-slow systems have been used to formalize the definition of a critical transition. As the next step, different viewpoints from stochastic dynamics were reviewed and their contributions to the prediction of critical transitions was discussed. For example, we have pointed out that variance increase immediately follows from well-known results of sample paths analysis or that P-bifurcations could act as a novel prediction mechanism. Then we focused on the variance as an indicator in the setup of normal forms and used analytical, numerical and geometric ideas to understand the increasing variance near a critical transition. Throughout our analysis we highlighted several challenges that arise in the modeling process including noise types (additive/multiplicative), problems with single sample paths as well as scaling laws for noise-induced phenomena. \\

We have not discussed further mechanisms and early warning signs that have been reported in applications:
\begin{itemize}
 \item[(a)] The change in spatial structure of a dynamical system can often be used as an indicator for an upcoming transition \cite{DonangeloFortDakosSchefferNes,RietkerkDekkerRuiterKoppel}. One could hope that bifurcation theory for pattern formation is applicable in this case \cite{SandstedeScheelWulff2,Hoyle}; in particular, reaction-diffusion PDEs might be the best starting point. The stochastic theory for SPDEs is much less developed \cite{GarciaOjalvoSancho} but statistical indicators are still expected to exist. 
 \item[(b)] We have focused primarily on the one-dimensional critical transitions (fold, transcritical, pitchfork). Although the pitchfork transition immediately gives results for the Hopf transition if the noise is only in the radial component, it does not capture its complete dynamics. The analysis of stochastic Hopf bifurcation is much more complicated than one-dimensional stochastic bifurcations \cite{KellerOchs,ArnoldRPB,BerglundGentz2}. We expect that the general analysis can be particularly complicated by noise correlated between the two fast variables.
 \item[(c)] Global bifurcations can induce drastic shifts in dynamical systems \cite{HomburgSandstede,Wiggins1}. In this respect, it becomes evident that we should also address critical transitions for iterated maps since they appear as Poincar\'{e} maps for differential equations; for example, it is well-known that critical slowing down occurs near a period-doubling bifurcation \cite{Hao}.  
 \item[(d)] Chaotic systems might provide special indicators that could be examined \cite{Schefferetal}. The generation of many chaotic attractors is preceded by well-analyzed bifurcation sequences \cite{GH,AlligoodSauerYorke}. Therefore it is conceivable that one might be able to modify or extend existing methods to yield early warning signs.
 \item[(e)] Fast-slow systems with three or more dimensions have not been discussed here. One example are fold bifurcations with two slow variables and one fast variable \cite{Guckenheimer7,Wechselberger} which occur generically on one-dimensional curves. Small oscillations can occur before a trajectory reaches a fold bifurcation and jumps to a far-away attractor. This behavior could be used as an indicator to predict a critical transition; a detailed review of the deterministic case in the context of mixed-mode oscillation can be found in \cite{KuehnMMO}. Stochastic folded nodes are discussed in \cite{BerglundGentzKuehn}. 
 \item[(f)] We have also not discussed the effect of noise on delay. The main point in this context is that small noise can reduce the deterministic delay effect discussed in Section \ref{sec:det2}; we refer the reader to \cite{Kuske,BerglundGentz6} and references therein for a more detailed discussion. However, let us note that it be very desirable to find early-warning signs before the delay-region i.e. calculating the precise jump time should be the second step of the mathematical analysis. 
\end{itemize}

We hope that the framework we reviewed and augmented in this paper also provides a better bridge between critical transitions in applications and the associated open mathematical challenges. It is expected that some new mathematical methods are going to be needed to address (a)-(f). Furthermore, we are fully aware that we have not maximized the results one can obtain from techniques presented here. For further results on normal forms, scaling of the variance and several applications see \cite{KuehnCT2}.\\ 

\textbf{Acknowledgements:} The author would like to thank John Guckenheimer for very valuable contributions to this paper. In particular, his ideas helped to significantly improve our definition of a critical transition point for deterministic fast-slow systems. Furthermore he provided many helpful suggestions for the deterministic part of the paper. The author also thanks the two anonymous referees for their helpful suggestions.

\end{document}